\documentclass[10pt]{amsart}

\usepackage{amsmath}
\usepackage{amssymb}
\usepackage{graphicx}

\newtheorem{theorem}{Theorem}[section]
\newtheorem{corollary}[theorem]{Corollary}
\newtheorem{lemma}[theorem]{Lemma}
\newtheorem{proposition}[theorem]{Proposition}
\theoremstyle{definition}
\newtheorem{definition}[theorem]{Definition}
\newtheorem{example}[theorem]{Example}
\newtheorem{remark}[theorem]{Remark}

\numberwithin{equation}{section}

\title[Optimality conditions for a class of MPEC problems]
{On some optimality conditions for a class of problems in mathematical
programming with equilibrium constraints}

\author[A. Uderzo]{Amos Uderzo}

\address[A. Uderzo]{Dept. of Mathematics and Applications, University
of Milano - Bicocca, Milano, Italy}
\email{{\tt amos.uderzo@unimib.it}}



\date{\today}

\newcommand{\R}{\mathbb R}
\newcommand{\N}{\mathbb N}

\newcommand{\X}{\mathbb X}
\newcommand{\Y}{\mathbb Y}
\newcommand{\Uball}{{\mathbb B}}
\newcommand{\Usfer}{{\mathbb S}}

\newcommand{\SAss}{\mathcal{A}}
\newcommand{\tepsilon}{\tilde{\epsilon}}

\newcommand{\dom}{{\rm dom}\, }
\newcommand{\graph}{{\rm gph}\,}
\newcommand{\epi}{{\rm epi}\,}

\newcommand{\nullv}{\mathbf{0}}

\newcommand{\clco}{{\rm clco}\, }

\newcommand{\cone}{{\rm cone}\, }
\newcommand{\bd}{{\rm bd}\, }
\newcommand{\inte}{{\rm int}\, }

\newcommand{\Argmin}{{\rm Argmin}}

\newcommand{\Equi}{{\mathcal S}{\mathcal E}}

\newcommand{\Kmap}{{\mathcal K}}

\newcommand{\MPEC}{{\rm MPEC}}

\newcommand{\VEP}{{\rm VEP}}

\newcommand{\mf}{\varsigma}

\newcommand{\Uplim}{{\rm Limsup}}

\newcommand{\parord}{\le_{{}_C}}

\newcommand{\Fsubd}{\widehat{\partial}}
\newcommand{\partialx}{\partial_x}

\newcommand{\stsl}[1]{|\nabla #1|}

\newcommand{\dcone}[1]{{#1}^{{}^\ominus}}

\newcommand{\ball}[2]{{\rm B}\left(#1;#2\right)}

\newcommand{\dist}[2]{{\rm dist}\left(#1;#2\right)}
\newcommand{\exc}[2]{{\rm exc}(#1;#2)}

\newcommand{\Ncone}[2]{{\rm N}(#1;#2)}

\newcommand{\Gder}[2]{{\rm D}#1(#2)}

\newcommand{\Proj}[2]{\Pi(#1;#2)}
\newcommand{\Ucone}[2]{{\rm N}^\flat(#1;#2)}

\newcommand{\Coder}[3]{{\rm D}^*#1(#2,#3)}

\begin{document}

\begin{abstract}
This paper considers mathematical programs, whose constraints are
expressed by a parameterized vector equilibrium problem. The latter
is a well recognized framework, which is able to cover multicriteria
optimization, vector variational inequalities and complementarity problems.
As the solutions to vector equilibrium problems are here intended in a
strong sense, the consequent MPEC problems result in a class still
little explored by the existing literature.
Some necessary optimality conditions for such programs are established
following a penalization approach. To derive and express these
conditions, concepts and tools of nonsmooth analysis are employed.
In treating equilibrium constraints, by techniques of variational
analysis some error bounds are obtained, which may be of independent
interest.
\end{abstract}

\maketitle




\section{Introduction and problem statement}

A mathematical program with equilibrium constraints, usually referred to as
MPEC, is an optimization problem essentially different from standard problems
of nonlinear programming with traditional equality and inequality constraints.
In MPEC, the objective function depends on two kinds of variables,
having different roles. Indeed, there is one kind of variables (the so-called
``state" or ``follower" variables) whose feasibility is determined by
the other kind (the ``control" or ``leader" variables).
More specifically, the feasible region of a MPEC involves equilibrium/optimality
conditions, expressed by variational inequalities or by other types of conditions,
for a lower-level problem of parametric optimization. As a matter of fact,
a MPEC can be regarded as a generalization of a bilevel program, in which the follower's decision
problem is modeled as a lower-level equilibrium problem. An interested
reader will find a flourishing literature devoted to MPECs (see, among the others,
\cite{Demp03,GfrYe20,LuPaRa96,KocOut04,LuPaRaWu96,Mord06b,OuKoZo98,SchSto99,Ye05}).
Reference sources for applications of MPECs in various contexts (especially
in engineering and economics) are, for instance, \cite[Section 7]{Demp03}
and \cite[Section 5]{KocOut04}.

The present paper deals with a specific class of MPECs, in which the equilibrium
constraint requires to consider the strong solutions to a vector equilibrium
problem parameterized by the leader variable.
The latter problem is a variant generalizing the equilibrium problem
as introduced by Blum and Oettli (\cite{BluOet94}) to the case of bifunctions
taking values in partially ordered vector spaces.
It has been proposed to subsume in a suitable framework such problems
as multicriteria optimization, vector variational inequalities and vector
complementarity problems (see \cite{Ansa00,BiHaSc97,DanHad96,Gong06}).
A typical phenomenon arising in such kind of issue is the lack
of a solution notion naturally emerging a priori. Instead, several solution notions
are often singled out on the base of the set of conditions that one
may require to be satisfied, ranging from the weak equilibrium concept
to the strong one, with various intermediate cases.
In the present paper, the investigations will focus exclusively on
strong solutions to a parameterized vector equilibrium problem.

In mathematical terms, the class of MPECs considered in the present paper
can be formalized as follows
$$
  \min\varphi(\xi,x) \quad \hbox{ subject to } \quad x\in\Equi(\xi),\
  \xi\in\Omega,
  \leqno (\MPEC)
$$
where the objective function $\varphi:\R^p\times\R^n\longrightarrow\R$ and the geometric
constraint $\Omega\subseteq\R^p$ are given problem data, $\Omega$ being closed,
while $\Equi:\R^p\rightrightarrows\R^n$ denotes the solution mapping
to the lower-level parametric equilibrium problem, namely
$$
  \hbox{ find $x\in\Kmap(\xi)$ such that } f(\xi,x,z)\in C,
  \quad\forall z\in\Kmap(\xi).  \leqno  (\VEP(\xi))
$$
The latter problem is defined by a mapping (parametric vector-valued bifunction)
$f:\R^p\times\R^n\times\R^n\longrightarrow\R^m$, with $\R^m$
being partially ordered by a (nontrivial) closed, convex
and pointed cone $C\subset\R^m$, and by the set-valued mapping
$\Kmap:\R^p\rightrightarrows\R^n$ modeling the feasible region of the
parameterized vector equilibrium problems. In other words, it is
$$
  \Equi(\xi)=\{x\in\R^n:\ x \hbox{ solves } (\VEP(\xi))\}.
$$
It is clear that in $(\MPEC)$ $x$ plays the role of follower variable,
whereas $\xi$ plays the role of leader variable.
As one immediately realizes, in general the characteristic constraint $x\in\Equi(\xi)$
defines only implicitly the set-valued mapping $\Equi$. In fact, in most
cases such a set-valued mapping can be hardly determined explicitly and
its graph may exhibit a bizarre behaviour that makes the feasible region
of a $(\MPEC)$ fairly complicated. This is the ``most distinctive"
feature of MPECs.
When by proper specialization of $f$, $(\VEP(\xi))$ becomes a vector variational
inequality or a vector complementarity problem, the corresponding
$(\MPEC)$, if reduced to a single-level standard program, is known to violate
most of the standard constraint qualifications.

\begin{remark}
The more general geometric constraint $(\xi,x)\in\Omega\subseteq
\R^p\times\R^n$ that one could consider for needs in a specific model design
can be easily subsumed in the format $(\MPEC)$, by introducing
the set-valued mapping $\mathcal{O}:\R^p\rightrightarrows\R^n$
such that $\graph\mathcal{O}=\Omega$, i.e. $\mathcal{O}(\xi)=
\{x\in\R^n:\ (\xi,x)\in\Omega\}$, and by replacing $\Kmap$ with
$\widetilde{\Kmap}:\R^p\rightrightarrows\R^n$, defined as being
$\widetilde{\Kmap}(\xi)=\Kmap(\xi)\cap\mathcal{O}(\xi)$.
\end{remark}

The main goal of the investigations exposed in the paper is to
derive necessary optimality conditions for $(\MPEC)$. This
task is undertaken by a penalization approach, implemented
by means of error bounds for a family of $(\VEP(\xi))$.
The reader should notice that, even though the statement
of $(\MPEC)$ seems to be the same as that considered in
\cite{KocOut04,LuPaRaWu96,Mord06b,Mord06b}, the generalized
equation $(\VEP(\xi))$ modeling the constraint system is
substantially different from that addressed in the aforementioned
works. In fact, strong vector equilibrium problems are
a specialization of the set-valued inclusion problem,
whose peculiar solution behaviour has been explored in
\cite{Uder19}.

The organization of the rest of this paper is as follows.
Section \ref{Sect:2} is devoted to set up the apparatus
of technical tools needed to implement the penalization approach
to optimality conditions here followed. These tools are mainly
borrowed from set-valued analysis and generalized differentiation
theory.
In Section \ref{Sect:3} the main results of the paper are established,
which are stationary conditions for $(\MPEC)$ expressed in terms of
widely employed nonsmooth analysis constructions, namely
Mordukhovich normals, subdifferential and coderivatives. On the way
to the optimality conditions, error bounds for the solution set to
parameterized vector equilibrium problems are obtained and possible
connections with solution stability for vector equilibrium problems
are outlined, which appear to be new.

\vskip.5cm


\section{Notation and preliminaries}    \label{Sect:2}

The notation employed in this paper is standard.
The acronyms l.s.c.,  u.s.c. and p.h. stand for lower semicontinuous,
upper semicontinuous and positively homogeneous, respectively.
Given a finite-dimensional Euclidean space $\X$, its inner product
is marked by $\langle\cdot,\cdot\rangle$, while $\nullv$ stands for its
null vector.
The closed  ball centered at an element $x\in\X$, with radius
$r\ge 0$, is denoted by $\ball{x}{r}$. In particular,
$\Uball=\ball{\nullv}{1}$, whereas $\Usfer$ stands for the unit sphere.
The distance of a point $x$ from a set $S$ is denoted by $\dist{x}{S}$, with
the convention that $\dist{x}{\varnothing}=+\infty$.
Given a subset $S\subseteq\X$, $\inte S$ denotes its interior,
$\bd S$ its boundary, whereas $\cone S$ its conical hull and
$\clco S$ its convex closure.
Moreover, given any $\epsilon\ge 0$, $\ball{S}{\epsilon}=\{x\in\X:\
\dist{x}{S}\le\epsilon\}$ indicates the (closed) $\epsilon$-enlargement of $S$.
Given two subsets $A$ and $B$ in the same space, $\exc{A}{B}=\sup_{a\in A}
\dist{a}{B}$ stands for the excess of $A$ over $B$.
Whenever $C\subseteq\X$ is a cone, by $\dcone{C}
=\{x\in\X:\ \langle c,x\rangle\le 0,\quad\forall c\in C\}$
its negative dual cone is denoted.
Given a function $\varphi:\X\longrightarrow\R\cup\{\pm\infty\}$,
by $[\varphi\le 0]=\varphi^{-1}([-\infty,0])$ its $0$-sublevel set is
denoted, whereas $[\varphi>0]=\varphi^{-1}((0,+\infty])$ denotes
the strict $0$-superlevel set of $\varphi$.
The symbol $\dom\varphi=\varphi^{-1}(\R)$ indicates the domain of
the function $\varphi$, while $\partial\varphi(x)$ stands for the subdifferential
of $\varphi$ at $x$ in the sense of convex analysis (a.k.a. Fenchel
subdifferential), with the convention $\partial\varphi(x)=\varnothing$
if $x\not\in\dom\varphi$. The normal cone to a set $S$ at $x$ in the
sense of convex analysis is denoted by $\Ncone{x}{S}$.
Given a mapping $g:\X\longrightarrow\Y$, between Euclidean spaces,
its (Fr\'echet) derivative at $\bar x\in\X$ is indicated by
$\Gder{g}{\bar x}$. The adjoint map to $\Gder{g}{\bar x}$ is marked by
$\Gder{g}{\bar x}^*$.
Given a set-valued mapping $F:\X\rightrightarrows\Y$,
$\dom F=\{x\in\X:\ F(x)\ne\varnothing\}$ and $\graph F=\{(x,y)\in
\X\times\Y:\ y\in\ F(x)\}$ denote the domain and the graph of $F$,
respectively. Further notations will be introduced in the sequel,
contextually to their use.

Throughout the paper, the following standing assumptions concerning
the constraining map $\Kmap$ will be maintained:
$$
  \graph\Kmap \hbox{ is closed and }\dom\Kmap=\R^p.
  \leqno (\SAss)
$$
Clearly, one among the implications of $(\SAss)$ is that $\Kmap$ takes
always (nonempty) closed values.

According to the proposed approach, let us introduce the following
auxiliary functions:
\begin{equation}     \label{eq:defni}
  \nu(\xi,x)=\sup_{z\in\Kmap(\xi)}\dist{f(\xi,x,z)}{C}
  =\exc{f(\xi,x,\Kmap(\xi))}{C},
\end{equation}
\begin{equation}    \label{eq:defmi}
  \mu(\xi,x)=\dist{x}{\Kmap(\xi)},
\end{equation}
and
\begin{equation}     \label{eq:defmf}
  \mf(\xi,x)=\nu(\xi,x)+\mu(\xi,x).
\end{equation}
From their very definition the above functions are expected to be
nonsmooth, in general. Yet,
by means of $\mf$ it is possible to provide a convenient functional
characterization of the graph and the values of $\Equi$.

\begin{remark}
By using the definitions in $(\ref{eq:defni})$, $(\ref{eq:defmi})$ and
$(\ref{eq:defmf})$ along with the fact that $C$ and $\Kmap(\xi)$,
for every $\xi\in\R^p$, are closed sets, it is readily seen that
$$
  \Equi(\xi)=\mf^{-1}(\xi,\cdot)(0)=[\mf(\xi,\cdot)\le 0]
  =[\nu(\xi,\cdot)\le 0]\cap[\mu(\xi,\cdot)\le 0],
  \quad\forall\xi\in\R^p
$$
and
\begin{equation}     \label{eq:grphEquichar}
\graph\Equi=[\mf\le 0].
\end{equation}
\end{remark}

Optimality conditions for $(\MPEC)$ will be investigated and
formulated by means of tools of variational analysis. This
area provides adequate mathematical resources for addressing
such troublesome constraint systems as those defined by
solution mappings to parametric variational problems.
Let us first recall a geometric qualification of sets, which
will be useful to treat separately the constraints $\xi\in\Omega$
and $x\in\Equi(\xi)$.
The property formalized below refers to a certain ``good" mutual
arrangements of several sets in space. It is a regularity property
which appeared under different names in different contexts, e.g.
in the convergence theory of projection methods for feasibility
problems and, as a qualification condition, in subdifferential
and normal calculus (see \cite{BauBor96,KrLuTh18}).

\begin{definition}[Subtransversality of sets]
A pair of subsets $S_1,\, S_2$ of an Euclidean space $\X$ is said to be
{\it subtransversal} at $\bar s\in S_1\cap S_2$ if there exist
positive constants $\alpha$ and $\delta$ such that
$$
 \left(S_1+(\alpha\rho)\Uball\right)\cap
 \left(S_2+(\alpha\rho)\Uball\right)\cap\ball{\bar s}{\delta}
 \subseteq (S_1\cap S_2)+\rho\Uball,
 \quad\forall \rho\in [0,\delta).
$$
\end{definition}

This property admits the following equivalent metric reformulation,
which will be exploited in the sequel: the pair  $S_1,\, S_2$ is
subtransversal at $\bar s$ iff there exist $\kappa,\, r>0$ such that
$$
  \dist{w}{S_1\cap S_2}\le\kappa\max\{\dist{w}{S_1},
  \dist{w}{S_2}\},\quad\forall w\in\ball{\bar s}{r}
$$
(see \cite[Theorem 1(ii)]{KrLuTh18}).
The geometric idea behind the above metric inequality should
be transparent: ``if you are close to both the sets of the pair,
then the intersection cannot be too far away" \cite{BauBor96}.
Since subtransversality will be employed as a qualification condition
in treating the subtle constraint system of $(\MPEC)$, it is
useful to recall situations, in which this property takes place.
If it is $\bar s\in\inte(S_1\cap S_2)$, then the pair $S_1,\, S_2$
is subtransversal at $\bar s$.
If $S_1$ and $S_2$ are closed and convex and $\nullv\in\inte(S_1-S_2)$,
the pair $S_1,\, S_2$ is subtransversal at each point $\bar s
\in S_1\cap S_2$.
In particular, a pair of polyhedral convex sets turns out to be subtransversal at
any intersection point (see \cite[Corollary 5.26]{BauBor96}).
Besides, from the metric reformulation it is clear that
the pair $S$ and $\X$ is subtransversal at any $\bar s\in S$,
whatever the set $S$ is.
Another sufficient condition for subtransversality will be formulated
in the next subsection by means of nonsmooth analysis constructions.

A further property, which will come into play, relates the covering
behaviour of mappings. A mapping $g:\X\longrightarrow\Y$ between
Euclidean spaces is said to open on $S\subseteq\X$ with a linear rate
$\alpha>0$ if
$$
  g(\ball{x}{r})\supseteq\ball{g(x)}{\alpha r},\quad\forall
  r>0,\ \forall x\in S.
$$
Openness with a linear rate (a.k.a. covering/surjection property)
is a crucial property equivalent to metric regularity,
whose various phenomenology has been deeply investigated and
largely employed in variational analysis (see \cite{Ioff00},
\cite[Chapter 1.2.3]{Mord06a},\cite[Chapther 3]{Mord18},
\cite[Chapter 9.G]{RocWet98}). In particular, a linear mapping
$\Lambda:\X\longrightarrow\Y$ between Euclidean spaces is open
on $S=\X$ iff it is onto, i.e. $\Lambda\X=\Y$.

\subsection{Elements of Set-Valued Analysis}

Variational methods often require some semicontinuity property
on involved functions to work. In the case of the auxiliary function
$\nu$, its lower semicontinuity can be obtained by combining
semicontinuity properties of $\Kmap$ and $f$. Recall that a
mapping $g:\X\longrightarrow\Y$ between Euclidean spaces, with $\Y$
partially ordered by a convex cone $C$, is said to be $C$-u.s.c.
at $x_0\in\X$ if for every neighbourhood $V$ of $g(x_0)$
there exists a neighbourhood $U$ of $x_0$ such that
$g(x)\in\ V-C$ for every $x\in U$.

\begin{lemma}[Lower semicontinuity of $\nu$]  \label{lem:lscni}
Let the function $\nu:\R^p\times\R^n\longrightarrow [0,+\infty]$
be defined as in $(\ref{eq:defni})$. If
\begin{itemize}

\item[(i)] $\Kmap:\R^p\rightrightarrows\R^n$ is l.s.c.;

\item[(ii)] $f:\R^p\times\R^n\times\R^n\longrightarrow\R^m$
is $C$-u.s.c.;
\end{itemize}
then $\nu$ is l.s.c. on $\R^p\times\R^n$.
\end{lemma}

\begin{proof}
According to \cite[Lemma 2.1]{Uder22}, by hypothesis (ii)
function $(\xi,x,z)\mapsto\dist{f(\xi,x,z)}{C}$ is l.s.c.
on $\R^p\times\R^n\times\R^n$. Thus, since the set-valued mapping
$\widetilde{\Kmap}:\R^p\times\R^n\rightrightarrows\R^n$, defined by
$\widetilde{\Kmap}(\xi,x)=\Kmap(\xi)$ is l.s.c. owing to hypothesis (i),
it is possible to invoke \cite[Lemma 17.29]{AliBor06}, which ensures
that the function given by
$$
  \nu(\xi,x)=\sup_{z\in\widetilde{\Kmap}(\xi,x)}\dist{f((\xi,x),z)}{C}
  =\sup_{z\in\Kmap(\xi)}\dist{f(\xi,x,z)}{C}
$$
is l.s.c. on $\R^p\times\R^n$, thereby completing the proof.
\end{proof}

\begin{lemma}[Lower semicontinuity of $\mu$]     \label{lem:lscmi}
Let the function $\mu:\R^p\times\R^n\longrightarrow [0,+\infty]$
be defined as in $(\ref{eq:defmi})$. If $\Kmap:\R^p
\rightrightarrows\R^n$ is u.s.c., then $\mu$ is l.s.c.
on $\R^p\times\R^n$.
\end{lemma}

\begin{proof}
Fix $(\bar\xi,\bar x)\in\R^p\times\R^n$ and take an arbitrary sequence
$(\xi_k,x_k)_k$ in $\R^p\times\R^n$ such that
$(\xi_k,x_k)\to (\bar\xi,\bar x)$ as $k\to\infty$. By upper semicontinuity
of $\Kmap$ at $\bar\xi$, fixed any $\epsilon>0$ there exists
$\delta_\epsilon\in (0,\epsilon)$ such that
$$
  \Kmap(\xi)\subseteq\inte\ball{\Kmap(\bar\xi)}{\epsilon},\quad\forall
  \xi\in\ball{\bar\xi}{\delta_\epsilon}.
$$
As $(\xi_k,x_k)_k$ converges to $(\bar\xi,\bar x)$, there exists
$k_\epsilon\in\N$ such that
$$
  (\xi_k,x_k)\in\ball{\bar\xi}{\epsilon}\times\ball{\bar x}{\epsilon},
  \quad\forall k\in\N, \ k\ge k_\epsilon.
$$
Consequently one finds
\begin{eqnarray*}
    \mu(\bar\xi,\bar x) &=& \dist{\bar x}{\Kmap(\bar\xi)}\le\dist{\bar x}{x_k}
    +\dist{x_k}{\Kmap(\xi_k)}+\exc{\Kmap(\xi_k)}{\Kmap(\bar\xi)} \\
    &\le& \epsilon+\mu(\xi_k,x_k)+\epsilon, \quad\forall k\in\N, \ k\ge k_\epsilon.
\end{eqnarray*}
This implies that $\displaystyle\liminf_{k\to\infty}\mu(\xi_k,x_k)\ge \mu(\bar\xi,\bar x)
-2\epsilon$, and so, by arbitrariness of $\epsilon>0$, it shows that $\mu$
is l.s.c. at $(\bar\xi,\bar x)$.
\end{proof}

The employment of a penalization technique to develop the present approach
requires  sufficient conditions for the set $\graph\Equi$ to be closed.
The next result serves the purpose.

\begin{corollary}[Closure of $\graph\Equi$]
With reference to a class of problem $(\VEP(\xi))$, with $\xi\in\R^p$,
suppose that:
\begin{itemize}

\item[(i)] $\Kmap:\R^p\rightrightarrows\R^n$ is continuous;

\item[(ii)] $f:\R^p\times\R^n\times\R^n\longrightarrow\R^m$
is $C$-u.s.c..
\end{itemize}
Then $\mf$ is l.s.c. and $\graph\Equi$ is closed.
\end{corollary}

\begin{proof}
Upon hypotheses (i) and (ii), it becomes
possible to combine Lemma \ref{lem:lscni} with Lemma \ref{lem:lscmi},
thereby obtaining that $\mf$ is l.s.c., as a sum of l.s.c. functions.
Thus the set $[\mf\le 0]$ is closed, so the thesis becomes a consequence
of the characterization in $(\ref{eq:grphEquichar})$.
\end{proof}

The auxiliary functions $\nu$ and $\mu$ also inherit some fruitful convexity
properties from convexity properties of $\Kmap$ and $f$, which are recalled
next.
Let $C\subseteq\Y$ be a convex cone in $\Y$. A mapping $g:\X\longrightarrow\Y$
between Euclidean spaces is said to be $C$-concave if for every $x_1,\, x_2
\in\X$ and $t\in [0,1]$, it is
$$
  tg(x_1)+(1-t)g(x_2)\parord g(tx_1+(1-t)x_2),
$$
or, equivalently,
$$
  g(tx_1+(1-t)x_2)-tg(x_1)-(1-t)g(x_2)\in C.
$$
A set-valued mapping $F:\X\rightrightarrows\Y$ between Euclidean spaces
is said to be:

\begin{itemize}

\item[(i)] convex on $\X$, if for every $x_1,\, x_2\in\X$ and $t\in [0,1]$, it is
$$
  F(tx_1+(1-t)x_2)\supseteq tF(x_1)+(1-t)F(x_2);
$$

\item[(ii)] concave on $\X$, if for every $x_1,\, x_2\in\X$ and $t\in [0,1]$, it is
$$
  F(tx_1+(1-t)x_2)\subseteq tF(x_1)+(1-t)F(x_2);
$$

\item[(iii)] affine on $\X$, if it is both convex and concave, i.e.
for every $x_1,\, x_2\in\X$ and $t\in [0,1]$, it is
$$
  F(tx_1+(1-t)x_2)=tF(x_1)+(1-t)F(x_2).
$$
\end{itemize}
The above notions are extensions of convexity/concavity for functions,
which are well recognized and largely exploited in optimization and variational
analysis.
Notice that in cases (i) and (iii) $F$ takes convex values, whereas
this may not be true in case (ii).

\begin{lemma}[Convexity of $\nu$]    \label{lem:convnu}
With reference to a problem $(\VEP(\xi))$, suppose that:

\begin{itemize}

\item[(i)] $\Kmap:\R^p\rightrightarrows\R^n$ is concave;

\item[(ii)] $f:\R^p\times\R^n\times\R^n\longrightarrow\R^m$
is $C$-concave;

\end{itemize}
Then, the function $\nu$ defined as in $(\ref{eq:defni})$
is convex on $\R^p\times\R^n$.
\end{lemma}

\begin{proof}
Let us recall first that on account of \cite[Lemma 2.5]{Uder22}
the $C$-concavity of $f$ implies the convexity of the function
$(\xi,x,z)\mapsto\dist{f(\xi,x,z)}{C}$.
Take arbitrary $(\xi_1,x_1)$, $(\xi_2,x_2)\in\R^p\times\R^n$ and
$t\in [0,1]$. By virtue of hypothesis (i), for any $z\in\Kmap
(t\xi_1+(1-t)\xi_2)\subseteq t\Kmap(\xi_1)+(1-t)\Kmap(\xi_2)$
there exist $z_1\in\Kmap(\xi_1)$ and $z_2\in\Kmap(\xi_2)$
such that $z=tz_1+(1-t)z_2$. Thus, by exploiting the convexity of
$(\xi,x,z)\mapsto\dist{f(\xi,x,z)}{C}$, one obtains
\begin{eqnarray*}
  \nu(t(\xi_1,x_1)+(1-t)(\xi_2,x_2)) &=& \sup_{z\in\Kmap
(t\xi_1+(1-t)\xi_2)}\dist{f(t(\xi_1,x_1,z)+(1-t)(\xi_2,x_2,z))}{C}  \\
  &\le& \sup_{z_1\in\Kmap(\xi_1) \atop z_2\in\Kmap(\xi_2)}
  \dist{f(t(\xi_1,x_1,z_1)+(1-t)(\xi_2,x_2,z_2))}{C}  \\
  &\le& \sup_{z_1\in\Kmap(\xi_1) \atop z_2\in\Kmap(\xi_2)}
  [t\dist{f(\xi_1,x_1,z_1)}{C}+(1-t)\dist{f(\xi_2,x_2,z_2)}{C}]  \\
  &=& t\nu(\xi_1,x_1)+(1-t)\nu(\xi_2,x_2),
\end{eqnarray*}
which shows the convexity of $\nu$.

\end{proof}

\begin{lemma}[Convexity of $\mu$]    \label{lem:convmu}
With reference to a problem $(\VEP(\xi))$, suppose that
the set-valued mapping $\Kmap:\R^p\rightrightarrows\R^n$
is convex. Then the function $\mu$ defined as in $(\ref{eq:defmi})$
is convex on $\R^p\times\R^n$.
\end{lemma}

\begin{proof}
Take arbitrary $\xi_1,\, \xi_2\in\R^p$, $x_1,\, x_2\in\R^n$, and
$t\in [0,1]$. By convexity of $\Kmap$, one has $\Kmap(t\xi_1+(1-t)\xi_2)
\supseteq t\Kmap(\xi_1)+(1-t)\Kmap(\xi_2)$, which implies
\begin{eqnarray*}
    \mu(t(\xi_1,x_1)+(1-t)(\xi_2,x_2)) &=& \dist{tx_1+(1-t)x_2}{\Kmap(t\xi_1+(1-t)\xi_2)}  \\
    &\le&  \dist{tx_1+(1-t)x_2}{t\Kmap(\xi_1)+(1-t)\Kmap(\xi_2)}  \\
    &=& \inf_{z_1\in\Kmap(\xi_1),z_2\in\Kmap(\xi_2)}
    \|tx_1+(1-t)x_2-(tz_1+(1-t)z_2)\|     \\
    &\le& \inf_{z_1\in\Kmap(\xi_1),z_2\in\Kmap(\xi_2)}
    [t\|x_1-z_1\|+(1-t)\|x_2-z_2\|]   \\
    &=& t\inf_{z_1\in\Kmap(\xi_1)}\|x_1-z_1\|+
    (1-t)\inf_{z_2\in\Kmap(\xi_2)}\|x_2-z_2\|  \\
    &=& t\mu(\xi_1,x_1)+(1-t)\mu(\xi_2,x_2).
\end{eqnarray*}
\end{proof}

\begin{remark}
It should be noticed that, as a consequence of the standing assumption
$(\SAss)$, it is $\dom\mu=\R^p\times\R^n$. So, in the light of the Lemma
\ref{lem:convmu} and a well-known property  of convex functions acting
in finite-dimensional spaces, whenever $\Kmap$ is convex, $\mu$ turns
out to be locally Lipschitz around each point of $\R^p\times\R^n$.
\end{remark}

\begin{corollary}[Convexity of $\Equi$]
With reference to a class of problems $(\VEP(\xi))$, with $\xi\in\R^p$,
suppose that:
\begin{itemize}

\item[(i)] $\Kmap:\R^p\rightrightarrows\R^n$ is affine;

\item[(ii)] $f:\R^p\times\R^n\times\R^n\longrightarrow\R^m$
is $C$-concave.

\end{itemize}
Then, $\mf:\R^\times\R^n\longrightarrow\R\cup\{\mp\infty\}$ and
$\Equi:\R^p\rightrightarrows\R^n$ are convex.
\end{corollary}

\begin{proof}
In the light of the characterization provided by $(\ref{eq:grphEquichar})$,
it suffices to observe that by Lemma \ref{lem:convnu} e Lemma
\ref{lem:convmu}, functions $\nu$ and $\mu$ are both convex and hence
so is their sum $\mf$. As a convex function, $\mf$ has convex
sublevel sets. Thus the thesis follows from $(\ref{eq:grphEquichar})$
\end{proof}

\vskip.5cm


\subsection{Elements of Nonsmooth Analysis}

Given a subset $S$ of an Euclidean space $\X$ and $\bar x\in S$,
the basic normal cone to $S$ at $\bar x$ is defined by
$$
   \begin{array}{ccc}
   \Ncone{\bar x}{S}= & \Uplim & \cone\left[x-\Proj{x}{S}\right], \\
   & \hbox{\scriptsize $x\to\bar x$ } &
   \end{array}
$$
where $\Uplim_{x\to\bar x}$ denotes the Painlev\'e-Kuratowski
outer/upper limit of a multifunction (see \cite[Chapter 4.B]{RocWet98})
and $\Proj{x}{S}=\{s\in S:\
\|x-s\|=\dist{x}{S}\}$ denotes the Euclidean projector of
$x$ to $S$. Recall that if $S$ is nonempty and closed, then
$\dom\Proj{\cdot}{S}=\X$, and if $S$ is nonempty, closed and convex,
then $\Proj{\cdot}{S}$ is single-valued.
As a direct consequence of the above definition, it is possible to derive
the following formula for the basic normal cone to the Cartesian product
of sets, which will be useful in the sequel: given $\bar x_1\in\Omega_1
\subseteq\R^p$ and $\bar x_2\in\Omega_2\subseteq\R^n$, it holds
$$
  \Ncone{(\bar x_1,\bar x_2)}{\Omega_1\times\Omega_2}=
  \Ncone{\bar x_1}{\Omega_1}\times \Ncone{\bar x_2}{\Omega_2}.
$$

According to the approach to nonsmooth analysis devised in
\cite{Mord06a,Mord18}, basic normals are the fundamental
elements on which further constructions rely, but they also appear
directly in the formulation of optimality conditions and qualification
conditions. Among other things, basic normals allow to express
a sufficient condition for subtransversality: if it is
\begin{equation}     \label{eq:subtransvnconecond}
   \Ncone{\bar s}{S_1}\cap[-\Ncone{\bar s}{S_2}]=
   \{\nullv\},
\end{equation}
the pair $S_1,\, S_2$ turns out to be subtransversal at $\bar s
\in S_1\cap S_2$ (see \cite[Theorem 2(v)]{KrLuTh18}).

Given a set-valued mapping $F:\X\rightrightarrows\Y$ between
Euclidean spaces, its basic coderivative at $(\bar x,\bar y)
\in\graph F$ is the multifunction $\Coder{F}{\bar x}{\bar y}:
\Y\rightrightarrows\X$ taking the values
$$
  \Coder{F}{\bar x}{\bar y}(v)=\{u\in\X:\ (u,-v)\in
  \Ncone{(\bar x,\bar y)}{\graph F}\},\qquad v\in\Y.
$$
Given a function $\psi:\X\longrightarrow\R\cup\{\pm\infty\}$
and $\bar x\in\psi^{-1}(\R)$, its basic (a.k.a. Mordukhovich)
subdifferential at $\bar x$ is defined by
$$
  \partial\psi(\bar x)=\{v\in\X:\ (v,-1)\in\Ncone{(\bar x,
  \psi(\bar x))}{\epi\psi}\},
$$
whereas its singular subdifferential at $\bar x$ is defined by
$$
  \partial^\infty\psi(\bar x)=\{v\in\X:\ (v,0)\in\Ncone{(\bar x,
  \psi(\bar x))}{\epi\psi}\}.
$$
The reader should notice that the above notation is not ambiguous
inasmuch as, whenever a function happens to be convex, its Mordukhovich
basic subdifferential and its subdifferential in the sense of convex
analysis do coincide. The same can be repeated for the normal cone
\`a la Mordukhovich and the normal cone in the sense of convex analysis.

In view of a subsequent employment, it is useful to recall that,
given a nonempty locally closed set $S\subseteq\X$,
since the function $x\mapsto\dist{x}{S}$ is Lipschitz continuous
on $\X$ with constant 1, then in the case $\bar x\in S$ one has
\begin{equation}    \label{eq:bsubddist}
  \partial\dist{\cdot}{S}(\bar x)=\Ncone{\bar x}{S}\cap\Uball
  \qquad\hbox{ and }\qquad
  \partial^\infty\dist{\cdot}{S}(\bar x)=\{\nullv\}
\end{equation}
(see \cite[Theorem 1.33(i)]{Mord18}),
whereas in the case $\bar x\not\in S$ one has
\begin{equation}    \label{eq:bsubdoosdist}
  \partial\dist{\cdot}{S}(\bar x)={\bar x-\Proj{\bar x}{S}
  \over\dist{\bar x}{S} }\subseteq\Usfer.
\end{equation}
(see \cite[Theorem 1.33(ii)]{Mord18}).

The basic subdifferential is known to enjoy a rich calculus.
For the purpose of the present analysis, it is to be recalled that
if $\psi_1,\, \psi_2:\X\longrightarrow\R\cup\{\pm\infty\}$ and
$\bar x\in\dom\psi_1\cap\dom\psi_2$ are such that the
qualification condition
\begin{equation}   \label{eq:singsubdqc}
    \partial^\infty\psi_1(\bar x)\cap\left(
    -\partial^\infty\psi_2(\bar x)\right)=\{\nullv\}
\end{equation}
is satisfied, then the following sum rule holds
$$
  \partial\left(\psi_1+\psi_2\right)(\bar x)\subseteq
  \partial\psi_1(\bar x)+\partial\psi_2(\bar x).
$$
One sees at once that the validity of $(\ref{eq:singsubdqc})$
is ensured in the case $\psi_1$ and $\psi_2$ forms a so-called
semi-Lipschitzian pair of functions, namely $\psi_1$
is l.s.c. around $\bar x$ and $\psi_2$ is locally Lipschitz
around the same point.

Another calculus rule that will be employed refers to the marginal
function $\mu_{\psi,G}$ associate with a function $\psi:\X\longrightarrow\R
\cup\{\pm\infty\}$ and a (closed-graph) set-valued mapping
$G:\X\rightrightarrows\Y$, i.e.
$$
  \mu_{\psi,G}(x)=\inf_{y\in G(x)}\psi(x,y).
$$
To express the basic subdifferential of $\mu_{\psi,G}$, one
needs to introduce the argminimum mapping $M:\X\rightrightarrows\Y$
$$
  M(x)=\{y\in\ G(x):\ \psi(x,y)=\mu_{\psi,G}(x)\}.
$$
Thus, according to \cite[Theorem 4.1(ii)]{Mord18}, if $M$ is locally bounded
around  $\bar x$ with $M(\bar x)\ne\varnothing$, and the qualification condition
\begin{equation}    \label{eq:subdmarmuqc}
   \partial^\infty\psi(\bar x,\bar y)\cap\left[
   -\Ncone{(\bar x,\bar y)}{\graph G}\right]=\{\nullv\}
\end{equation}
is satisfied, then the following outer estimate holds true
\begin{equation}    \label{in:subdmarmu}
   \partial\mu_{\psi,G}(\bar x)\subseteq
   \bigcup \left\{x^*+\Coder{G}{\bar x}{\bar y}(y^*):\
   (x^*,y^*)\in\partial\psi(\bar x,\bar y),\ \bar y\in M(\bar x)\right\}.
\end{equation}

In the context of the present analysis, the rule expressed in
$(\ref{in:subdmarmu})$ enables one to provide an useful outer
estimate of the basic subdifferential of $\mu$.

\begin{lemma}     \label{lem:bsubdmi}
Let $\Kmap:\R^p\rightrightarrows\R^n$ be l.s.c. and let $(\bar\xi,\bar x)
\in\R^p\times\R^n$. Then for the function $\mu$ defined as in $(\ref{eq:defmi})$
the following estimate holds
\begin{equation}     \label{in:bsubdifmi}
    \partial\mu(\bar\xi,\bar x)\subseteq
    \bigcup_{\bar z\in\Proj{\bar x}{\Kmap(\bar\xi)}}
    \left[\Coder{\Kmap}{\bar\xi}{\bar z}(\Uball)\times\Uball\right].
\end{equation}
\end{lemma}

\begin{proof}
Since it is
$$
  \mu(\xi,x)=\inf_{z\in\Kmap(\xi)}\|x-z\|=
  \inf_{z\in\widetilde{\Kmap}(\xi,x)}\widetilde{\psi}(\xi,x,z)
$$
with $\widetilde{\Kmap}:\R^p\times\R^n\rightrightarrows\R^n$ defined
by $\widetilde{\Kmap}(\xi,x)=\Kmap(\xi)$ and $\widetilde{\psi}:\R^p\times\R^n\times\R^n
\longrightarrow\R$ defined by $\widetilde{\psi}(\xi,x,z)=\|x-z\|$,
the idea is to apply formula $(\ref{in:subdmarmu})$ with $G=\widetilde{\Kmap}$,
$\psi=\widetilde{\psi}$, and $\X=\R^p\times\R^n$, $\Y=\R^n$.
Remember that under assumption $(\SAss)$ $\graph\Kmap$ is closed
and observe that, in the current setting, it is
$$
  M(\xi,x)=\Proj{x}{\Kmap(\xi)}.
$$
So, it is $\dom M=\R^p\times\R^n$. Besides, since $\Kmap$ is
supposed to be l.s.c., then the set-valued mapping $(\xi,x)\leadsto\Proj{x}{\Kmap(\xi)}$
is locally bounded around $(\bar\xi,\bar x)$. To see this, take any
$\bar z\in\Proj{\bar x}{\Kmap(\bar \xi)}$ and fix $r_0>0$ in such
a way that $r_0>\|\bar x-\bar z\|$. Since $\bar z\in\inte
\ball{\bar x}{r_0}\cap\Kmap(\bar\xi)\ne\varnothing$, by lower
semicontinuity of $\Kmap$  at $\bar\xi$ there exists $\delta_0
\in(0,r_0)$ such that
$$
  \inte\ball{\bar x}{r_0}\cap\Kmap(\xi)\ne\varnothing,
  \quad\forall\xi\in\ball{\bar x}{\delta_0}.
$$
It follows
$$
  \Proj{\bar x}{\Kmap(\xi)}\subseteq\ball{\bar x}{r_0},
  \quad\forall \xi\in\ball{\bar\xi}{\delta_0}.
$$
As each set-valued mapping $x\leadsto\Proj{x}{\Kmap(\xi)}$ is
Lipschitz continuous ($\xi$ being fixed) with constant $1$,
it holds
$$
  \Proj{x}{\Kmap(\xi)}\subseteq \Proj{\bar x}{\Kmap(\xi)}+
  \|x-\bar x\|\Uball.
$$
Therefore, it results in
$$
  \Proj{x}{\Kmap(\xi)}\subseteq\ball{\bar x}{r_0}+\delta_0\Uball
  \subseteq\ball{\bar x}{2r_0},\quad\forall
  (\xi,x)\in\ball{\bar\xi}{\delta_0}\times \ball{\bar x}{\delta_0},
$$
meaning that the set-valued mapping $(\xi,x)\leadsto\Proj{x}{\Kmap(\xi)}$
is locally bounded around $(\bar\xi,\bar x)$.
Since $\widetilde{\psi}$ is clearly Lipschitz continuous
(with constant $1$), the qualification condition $(\ref{eq:subdmarmuqc})$
is fulfilled. Moreover, as $\widetilde{\psi}$ is constant
with respect to $\xi$, one can write
$$
  \partial\widetilde{\psi}(\bar\xi,\bar x,\bar z)\subseteq
  \{\nullv\}\times\Uball\times\Uball.
$$
On the other hand, according to the coderivative definition,
it is
$$
   \Coder{\widetilde{\Kmap}}{(\bar\xi,\bar x)}{\bar z}(v)=
   \{(\xi^*,x^*)\in\R^p\times\R^n:\ (\xi^*,x^*,-v)\in
   \Ncone{(\bar\xi,\bar x),\bar z}{\graph\widetilde{\Kmap}}\}.
$$
Clearly, the fact that $\widetilde{\Kmap}$ is constant with
respect to $x$ implies
$$
  \graph\widetilde{\Kmap}=(\graph\Kmap)\times\R^n.
$$
Therefore, according to \cite[Proposition 1.4]{Mord18}
it results in
$$
  \Ncone{(\bar\xi,\bar x),\bar z}{\graph\widetilde{\Kmap}}=
  \Ncone{(\bar\xi,\bar z)}{\graph{\Kmap}}\times\Ncone{\bar x}{\R^n}=
  \Ncone{(\bar\xi,\bar z)}{\graph{\Kmap}}\times\{\nullv\}.
$$
On account of the last equality, one obtains
$$
  \Coder{\widetilde{\Kmap}}{(\bar\xi,\bar x)}{\bar z}(v)=
  \Coder{\Kmap}{\bar\xi}{\bar z}(v)\times\{\nullv\}.
$$
By applying the outer estimate $(\ref{in:subdmarmu})$ with
the above elements, one finds
\begin{eqnarray*}
 \partial\mu(\bar\xi,\bar x) &\subseteq &
    \bigcup_{\bar z\in\Proj{\bar x}{\Kmap(\bar\xi)}}
    \left\{(\xi^*,x^*)+\Coder{\widetilde{\Kmap}}{(\bar\xi,\bar x)}{\bar z}(z^*):\
   (\xi^*,x^*,z^*)\in\partial\widetilde{\psi}(\bar\xi,\bar x,\bar z)\right\}  \\
   &\subseteq& \bigcup_{\bar z\in\Proj{\bar x}{\Kmap(\bar\xi)}}
   \left\{(\nullv,u)+\left(\Coder{\Kmap}{\bar\xi}{\bar z}(v)\times\{\nullv\}\right):\
   (u,v)\in\Uball\times\Uball\right\}  \\
    &=& \bigcup_{\bar z\in\Proj{\bar x}{\Kmap(\bar\xi)}}
   \left\{(w,u)\in\R^p\times\R^n:\ w\in\Coder{\Kmap}{\bar\xi}{\bar z}(v),\
   (u,v)\in\Uball\times\Uball\right\},
\end{eqnarray*}
which leads to inclusion $(\ref{in:bsubdifmi})$, thereby completing the proof.
\end{proof}

\begin{remark}
Satisfactory formulae for basic and other limiting subgradients
of the function $\mu$ have been already established in the variational analysis
literature.
In particular, the upper estimates in \cite[Theorem 4.9]{MorNam05} seem to be close
to inclusion $(\ref{in:bsubdifmi})$. Nonetheless, it is worth remarking that
formula $(\ref{in:bsubdifmi})$ refers to both the possible cases $\bar x\in\Kmap(\bar\xi)$
and $\bar x\not\in\Kmap(\bar\xi)$. Moreover, since it has been derived in a
much more special setting, the argument in the proof here proposed avoids
the well-posedness condition imposed in the statement of Theorem 4.9.
\end{remark}

As the constrained problem $(\MPEC)$ will be reduced to an unconstrained one,
it is useful to recall that, according to a basic optimization principle,
whenever $\bar x\in\dom\psi$ is a local unconstrained minimizer of $\psi:\X
\longrightarrow\R\cup\{\pm\infty\}$, it must hold
\begin{equation}   \label{in:buncoptcond}
  \nullv\in\partial\psi(\bar x).
\end{equation}
Whenever $\psi:\R^p\times\R^n\longrightarrow\R\cup\{\mp\infty\}$ and
$(\bar\xi,\bar x)\in\R^p\times\R^n$ are given, the notation $\partialx
\psi(\bar\xi,\bar x)$ is sometimes used to indicate the partial subdifferential
of $\psi$ with respect to $x$, calculated at $(\bar\xi,\bar x)$, i.e.
$\partial\psi(\bar\xi,\cdot)(\bar x)\subseteq\R^n$.

\vskip.5cm


\section{Optimality conditions via a penalization approach}    \label{Sect:3}

Following a geometric approach developed in \cite[Chapter 5.2]{Mord06b},
let us start noticing that a problem of $(\MPEC)$ can be equivalently
reformulated as
\begin{equation}     \label{pr:geomreform}
   \min\varphi(\xi,x) \quad \hbox{ subject to } \quad (\xi,x)\in
   \widetilde{\Omega}\cap\graph\Equi,
\end{equation}
where $\widetilde{\Omega}=\Omega\times\R^n$. It is well known that,
under a locally Lipschitz assumption of $\varphi$, one can convert
the geometric constraint in $(\ref{pr:geomreform})$ into functional terms
by a well-known penalization technique, relying on the existence of
a residual function measuring the constraint violation. The
implementation of this general principle for problems of the form
$(\MPEC)$ runs as follows.

\begin{proposition}[Basic penalization principle]    \label{pro:Penpri}
With reference to a problem $(\MPEC)$, let $(\bar\xi,\bar x)\in\R^p\times\R^n$.
Suppose that:
\begin{itemize}

\item[(i)] $\varphi$ is Lipschitz continuous on $\ball{(\bar\xi,\bar x)}{\rho}$,
for some $\rho>0$, with constant $\ell_\varphi>0$;

\item[(ii)] there exist $\sigma:\ball{(\bar\xi,\bar x)}{\rho}\longrightarrow
[0,+\infty)$, vanishing on $\widetilde{\Omega}\cap\graph\Equi\cap
\ball{(\bar\xi,\bar x)}{\rho}$, and $\tau>0$ such that
$$
  \dist{(\xi,x)}{\widetilde{\Omega}\cap\graph\Equi}\le\tau\sigma(\xi,x),
  \quad\forall (\xi,x)\in\ball{(\bar\xi,\bar x)}{\rho};
$$

\item[(iii)] $\graph\Equi$ is closed and $(\MPEC)$ admits a solution.
\end{itemize}
Then, for any $\lambda>\ell_\varphi\tau$ it holds
\begin{eqnarray*}
  \Argmin\{\varphi(\xi,x):\ (\xi,x)\in \widetilde{\Omega}\cap\graph\Equi\cap
  \ball{(\bar\xi,\bar x)}{\rho}\} \\
  = \Argmin\{\varphi(\xi,x)+\lambda\sigma(\xi,x):\ (\xi,x)
  \in\ball{(\bar\xi,\bar x)}{\rho}\}.
\end{eqnarray*}
\end{proposition}

\begin{proof}
It suffices to apply, for instance, \cite[Theorem 6.8.1]{FacPan03} with
$\theta=\varphi$, $X=\ball{(\bar\xi,\bar x)}{\rho}$, $W=\widetilde{\Omega}\cap\graph\Equi$
and $S=\widetilde{\Omega}\cap\graph\Equi\cap\ball{(\bar\xi,\bar x)}{\rho}$.
\end{proof}

In order to treat the characteristic constraint $x\in\Equi(\xi)$ according
to this penalization approach, one may use the auxiliary function $\mf$
introduced in $(\ref{eq:defmf})$. The situation in which $\mf$ works as
a merit function for the constraint $\xi\in\Equi(\xi)$ is captured by
the concept of error bounds. More specifically,
a uniform error bound around $(\bar\xi,\bar x)\in\graph\Equi$ is
said to hold for a class of problems $(\VEP(\xi))$ if there exist $\delta,\,
\gamma>0$ such that
\begin{equation}     \label{in:unierboVEP}
   \dist{x}{\Equi(\xi)}\le{\mf(\xi,x)\over\gamma},\quad\forall
   (\xi,x)\in\ball{\bar\xi}{\delta}\times\ball{\bar x}{\delta}.
\end{equation}

Upon the occurrence of a uniform error bound the study of necessary
optimality conditions for $(\MPEC)$ can be reduced to those for
an unconstrained problem adequately penalized.

\begin{proposition}    \label{pro:constred}
Let $(\bar\xi,\bar x)\in\widetilde{\Omega}\cap\graph\Equi$ be a local solution
to $(\MPEC)$. Suppose that:

\begin{itemize}

\item[(i)] $\varphi$ is locally Lipschitz around $(\bar\xi,\bar x)$;

\item[(ii)] $\graph\Equi$ is closed;

\item[(iii)] $\widetilde{\Omega}$ and $\graph\Equi$ are subtransversal at
$(\bar\xi,\bar x)$;

\item[(iv)] an uniform error bound for problems $(\VEP(\xi))$
around $(\bar\xi,\bar x)$ as in $(\ref{in:unierboVEP})$ holds.

\end{itemize}
Then, there exists $\lambda>0$ such that the pair $(\bar\xi,\bar x)$
is a local unconstrained solution of the problem
\begin{equation}     \label{P:unpenal}
   \min\,\left[\varphi(\xi,x)+\lambda\left(\dist{\xi}{\Omega}+
   {\mf(\xi,x)\over\gamma}\right)\right].
\end{equation}
\end{proposition}

\begin{proof}
Since $(\bar\xi,\bar x)$ is a local solution to $(\MPEC)$, there exists
$\rho>0$ such that $(\bar\xi,\bar x)\in\Argmin\{\varphi(\xi,x):\ (\xi,x)
\in \widetilde{\Omega}\cap\graph\Equi\cap \ball{(\bar\xi,\bar x)}{\rho}\}$.
By hypothesis (i), up to a reduction in value of $\rho$, one can assume
that $\varphi$ is Lipschitz continuous on $\ball{(\bar\xi,\bar x)}{\rho}$,
with constant $\ell_\varphi>0$. Thus, it is possible to apply Proposition
\ref{pro:Penpri}, with
$$
  \sigma(\xi,x)=\dist{(\xi,x)}{\widetilde{\Omega}\cap\graph\Equi} \quad
  \hbox{ and }\quad  \tau=1.
$$
According to it, for any $\ell>\ell_\varphi$ it is true that
$(\bar\xi,\bar x)$ locally solves the unconstrained problem
\begin{equation}      \label{p:unconminpen}
   \min\ [\varphi(\xi,x)+\ell\dist{(\xi,x)}{\widetilde{\Omega}\cap\graph\Equi}].
\end{equation}
Since $\widetilde{\Omega}$ and $\graph\Equi$ are subtransversal at $(\bar\xi,\bar x)$
(hypothesis (iii)), there exist $\kappa>0$ and $r>0$ such that
\begin{eqnarray*}
  \dist{(\xi,x)}{\widetilde{\Omega}\cap\graph\Equi} &\le &\kappa
  [\dist{(\xi,x)}{\widetilde{\Omega}}+\dist{(\xi,x)}{\graph\Equi}], \\
  & & \forall (\xi,x)\in\ball{\bar\xi}{r}\times
  \ball{\bar x}{r}.
\end{eqnarray*}
Consequently, $(\bar\xi,\bar x)$ turns out to be a local solution
to the following unconstrained problem
$$
  \min\left[\varphi(\xi,x)+\ell\kappa\left(
  \dist{(\xi,x)}{\widetilde{\Omega}}+\dist{(\xi,x)}{\graph\Equi}
  \right)\right].
$$
Now, it is readily seen that
$$
  \dist{(\xi,x)}{\widetilde{\Omega}}=\dist{\xi}{\Omega},\quad\forall
  (\xi,x)\in\R^p\times\R^n.
$$
Besides, observe that, taken an arbitrary $(\xi,x)\in\R^p\times\R^n$,
if $\xi\in\dom\Equi$, then
\begin{eqnarray*}
   \dist{(\xi,x)}{\graph\Equi} &=& \inf\{\|(\hat\xi,\hat x)
   -(\xi,x)\|: \hat x\in\Equi(\hat\xi),\ \hat\xi\in\dom\Equi\} \\
   &\le& \inf\{\|(\xi,\tilde x)-(\xi,x)\|: \tilde x\in\Equi(\xi)\} \\
   &=& \inf\{\|\tilde x-x)\|: \tilde x\in\Equi(\xi)\}=\dist{x}{\Equi(\xi)}.
\end{eqnarray*}
If $\xi\not\in\dom\Equi$, since it is $(\bar\xi,\bar x)\in\graph\Equi
\ne\varnothing$, whereas $\Equi(\xi)=\varnothing$, one trivially
has
$$
   \dist{(\xi,x)}{\graph\Equi}\le +\infty=\dist{x}{\varnothing}=
   \dist{x}{\Equi(\xi)}.
$$
Such estimates allow one to deduce that $(\bar\xi,\bar x)$ is
also a local unconstrained minimizer on the problem
$$
  \min\left[\varphi(\xi,x)+\ell\kappa\left(
  \dist{\xi}{\Omega}+\dist{x}{\Equi(\xi)}\right)\right].
$$
By taking into account the validity of the error bound in $(\ref{in:unierboVEP})$,
this fact proves the assertion in the thesis, with $\lambda=
\ell\kappa$.
\end{proof}

In consideration of the dramatic role played by error bounds in the
current approach, the next lemma provides a sufficient subdifferential condition
for error bounds related to a family of problems $(\VEP(\xi))$. It can be
regarded as a parameterized version of \cite[Theorem 3.14]{Uder22}.
Its proof, which follows the argument exploited for the aforementioned
result, is given for the sake of completeness.
In its formulation, given $(\xi,x)\in\R^p\times\R^n$ it is convenient to
employ the following notation for indicating the unit truncation map
${\rm N}^\flat(\cdot,\Kmap(\cdot)):\R^n\times\R^p\rightrightarrows
\R^n$ of basic normal cones:
\begin{eqnarray}    \label{eq:deftruncmap}
   \Ucone{x}{\Kmap(\xi)}=\left\{\begin{array}{ll}
    \Ncone{x}{\Kmap(\xi)}\cap\Uball, & \hbox{ if } x\in\Kmap(\xi), \\
    \\
    \displaystyle{x-\Proj{x}{\Kmap(\xi)}\over\dist{x}{\Kmap(\xi)}},
     & \hbox{ if } x\not\in\Kmap(\xi).
   \end{array}\right.
\end{eqnarray}
Recall that a subset $S\subseteq\R^m$ is called $C$-bounded provided that
its subset $S\backslash C$ is (metrically) bounded.

\begin{lemma}[Parametric error bound]    \label{lem:erbo1}
Let $(\bar\xi,\bar x)\in\graph\Equi$ and $\rho>0$.
With reference to problem $(\VEP(\xi))$, where $\xi\in\ball{\bar\xi}{\rho}$,
suppose that:

\begin{itemize}

\item[(i)] function $x\mapsto f(\xi,x,z)$ is $C$-u.s.c., for every $z\in\Kmap(\xi)$;

\item[(ii)] there exists $x_{0,\xi}\in\Kmap(\xi)$ such that
$f(\xi,x_{0,\xi},\Kmap(\xi))$ is $C$-bounded;

\item[(iii)] there exists $\gamma>0$ such that
$$
  \left[\partial_x\nu(\xi,x)+\Ucone{x}{\Kmap(\xi)}\right]\cap
  \gamma\Uball=\varnothing,\quad\forall x\in\R^n
  \backslash\Equi(\xi).
$$
\end{itemize}
Then, $\Equi(\xi)\ne\varnothing$ and it holds
\begin{equation}    \label{in:erbo1}
   \dist{x}{\Equi(\xi)}\le{\mf(\xi,x)\over \gamma},
   \quad\forall x\in\R^n.
\end{equation}
\end{lemma}

\begin{proof}
Fix $\xi\in\ball{\bar\xi}{\rho}$.
Observe first that function $x\mapsto\nu(\xi,x)$ is l.s.c.
on $\R^n$ owing to hypothesis (i). This follows from Lemma \ref{lem:lscni},
because $\Kmap$, taking the constant value $\Kmap(\xi)$
($\xi$ being fixed), is a l.s.c. set-valued mapping
(otherwise,  see \cite[Remark 2.2]{Uder22}).
Moreover function $x\mapsto\mu(\xi,x)$ is Lipschitz continuous on
$\R^n$. Therefore, function $x\mapsto\mf(\xi,x)$ is l.s.c.
on $\R^n$, so the set $[\mf(\xi,\cdot)>0]$
turns out to be open.
Hypothesis (ii) implies that the set $[\mf(\xi,\cdot)<+\infty]$
is nonempty. Thus, according to \cite[Proposition 1, Chapter 3]{Ioff00}
such circumstances ensure the validity of the estimate
\begin{equation}     \label{in:stslFsubest}
  \inf_{x\in[\mf(\xi,\cdot)>0]}\stsl{\mf(\xi,\cdot)}(x)\ge
 \inf_{x\in[\mf(\xi,\cdot)>0]}\dist{\nullv}{\Fsubd_x\mf(\xi,x)},
\end{equation}
where
$$
 \stsl{\mf(\xi,\cdot)}(x_0)=\left\{\begin{array}{ll}
  0, & \hbox{ if $x_0$ is a local minimizer of } \mf(\xi,\cdot)  \\
  \displaystyle\limsup_{u\to x}
  {\mf(\xi,x_0)-\mf(\xi,x)\over\|x_0-x\|}, & \hbox{ otherwise,}
  \end{array}\right.
$$
is the strong slope of $\mf(\xi,\cdot)$ at $x_0$ and
$$
  \Fsubd_x\mf(\xi,x_0)=\left\{v\in\R^n:\ \liminf_{x\to x_0}
  {\mf(\xi,x)-\mf(\xi,x_0)-\langle v,x-x_0\rangle\over\|x-x_0\|}
  \ge 0\right\}
$$
denotes the partial Fr\'echet subdifferential with respect to $x$
of $\mf$ at $(\xi,x_0)$. As for any $x_0\in\R^n$ the limiting representation
$$
  \begin{array}{ccc}
  \partial_x\mf(\xi,x_0)= & \Uplim & \Fsubd_x\mf(\xi,x), \\
  & \hbox{\scriptsize $x\stackrel{\mf(\xi,\cdot)}{\longrightarrow} x_0$} &
  \end{array}
$$
holds (see \cite[Theorem 1.28]{Mord18},
one has $\Fsubd_x\mf(\xi,x_0)\subseteq\partial_x\mf(\xi,x_0)$. From inclusion
$(\ref{in:stslFsubest})$ it follows
$$
  \inf_{x\in[\mf(\xi,\cdot)>0]}\stsl{\mf(\xi,\cdot)}(x)\ge
  \inf_{x\in[\mf(\xi,\cdot)>0]}\dist{\nullv}{\partial_x\mf(\xi,x)}.
$$
Consequently,
by taking into account that the functions $\nu(\xi,\cdot)$ and $\mu(\xi,\cdot)$
forms a semi-Lipschitzian pair around $\bar x$, by virtue of the
sum rule and formulae $(\ref{eq:bsubddist})$ and $(\ref{eq:bsubdoosdist})$,
one obtains
$$
  \partial_x\mf(\xi,x)\subseteq\partial_x\nu(\xi,x)+
  \partial_x\mu(\xi,x)\subseteq\partial_x\nu(\xi,x)+\Ucone{x}{\Kmap(\xi)}.
$$
Thus, on account of hypothesis (iii), one obtains that
$$
   \inf_{x\in[\mf(\xi,\cdot)>0]}\stsl{\mf(\xi,\cdot)}(x)\ge\gamma,\quad
   \forall x\in\R^n\backslash\Equi(\xi).
$$
The last inequality, along with the lower semicontinuity of $\mf(\xi,\cdot)$
and the fact that $[\mf(\xi,\cdot)<+\infty]\ne\varnothing$,
allows one to invoke a well-known general error bound condition
valid in complete metric spaces (see, for instance \cite[Proposition 3.1]{Uder22})
which enables to achieve both the assertions in the thesis.
\end{proof}

\begin{remark}[Solution stability of strong vector equilibrium problems]
Lemma \ref{lem:erbo1} can be read as a solution stability result for a
parameterized family of vector equilibrium problems. The first assertion
indeed speaks about the local solvability (in the strong sense)
of problems $(\VEP(\xi))$, for $\xi$
varying around $\bar\xi$, saying that $\bar\xi\in\inte\dom\Equi$. Moreover,
by taking $x=\bar x$ in $(\ref{in:erbo1})$, one obtains
\begin{equation}   \label{in:lscEqui}
  \ball{\bar x}{{\mf(\xi,\bar x)\over\gamma}}\cap\Equi(\xi)\ne\varnothing,
  \quad\forall\xi\in\ball{\bar\xi}{\rho}.
\end{equation}
This leads to a form of quantitative lower semicontinuity of $\Equi$
at $(\bar\xi,\bar x)$. In particular, whenever the function $\xi\mapsto
\mf(\xi,\bar x)$ happens to be calm from above at $\bar\xi$, i.e.
there exists $\beta_\mf>0$ such that
$$
  \mf(\xi,\bar x)\le\beta_\mf\|\xi-\bar\xi\|, \quad\forall\xi\in
  \ball{\bar\xi}{r},
$$
for some $r>0$, then the inequality $(\ref{in:lscEqui})$ implies that
the set-valued mapping $\Equi$ is Lipschitz l.s.c. at $(\bar\xi,\bar x)$
in the sense of \cite[Chapter 1.5]{KlaKum02}.
Besides, whenever each function $\xi\mapsto\mf(\xi,x)$ is locally Lipschitz
around $\bar\xi$, with the same constant $\ell_\mf$ for every $x\in
\ball{\bar x}{\delta}$, then the inequality $(\ref{in:erbo1})$ yields
\begin{eqnarray*}
   \sup_{x\in\Equi(\xi_2)\cap\ball{\bar x}{\delta}}\dist{x}{\Equi(\xi_1)}
   &\le& \sup_{x\in\Equi(\xi_2)\cap\ball{\bar x}{\delta}}\gamma^{-1}
   [\mf(\xi_1,x)-\mf(\xi_2,x)+\mf(\xi_2,x)] \\
   &\le& \gamma^{-1}\ell_\mf\|\xi_1-\xi_2\|+\gamma^{-1}\sup_{x\in\Equi(\xi_2)\cap\ball{\bar x}{\delta}}
   \mf(\xi_2,x) \\
   &=& \gamma^{-1}\ell_\mf\|\xi_1-\xi_2\|, \quad\forall \xi\in\ball{\bar\xi}{r},
\end{eqnarray*}
for some $r>0$. Such an inequality means that $\Equi$ has the Aubin
property (equivalently, it is Lipschitz-like) around $(\bar\xi,\bar x)$
(see \cite{Ioff00},\cite[Chapter 1.2.2]{Mord06a},\cite[Chapter 9.F]{RocWet98}).
In turn, calmness and Lipschitz continuity properties of $\mf$
can be obtained by proper assumptions on $\Kmap$ and $f$.
\end{remark}

On the base of the previous preparatory results, one is in a position
to establish a first necessary optimality condition for $(\MPEC)$.

\begin{theorem}[General necessary optimality condition]   \label{thm:nocond1}
Let $(\bar\xi,\bar x)\in\graph\Equi$, with $\bar\xi\in\Omega$ be a local solution
to $(\MPEC)$, and let $\rho>0$. Suppose that:

\begin{itemize}

\item[(i)] $\varphi$ is locally Lipschitz around $(\bar\xi,\bar x)$;

\item[(ii)] $\Kmap$ is continuous;

\item[(iii)] $\widetilde{\Omega}$ and $\graph\Equi$ are subtransversal at
$(\bar\xi,\bar x)$;

\item[(iv)] $f:\R^p\times\R^n\times\R^n\longrightarrow\R^m$
is $C$-u.s.c.;

\item[(v)] the qualification condition $-\partial^\infty
\nu(\bar\xi,\bar x)\cap\left(\Coder{\Kmap}{\bar\xi}{\bar x}
(\Uball)\times\Uball\right)=\{\nullv\}$ holds;

\item[(vi)] for every $\xi\in\ball{\bar\xi}{\rho}$ there exists
$x_{0,\xi}\in\Kmap(\xi)$ such that $f(\xi,x_{0,\xi},\Kmap(\xi))$
is $C$-bounded;

\item[(vii)] there exists $\gamma>0$ such that
\begin{equation}      \label{eq:erbosubdgammacond}
  \left[\partialx\nu(\xi,x)+\Ucone{x}{\Kmap(\xi)}\right]\cap
  \gamma\Uball=\varnothing,\quad\forall x\in\R^n\backslash
  \Equi(\xi),\ \forall\xi\in\ball{\bar\xi}{\rho}.
\end{equation}
\end{itemize}
Then, there exists $\lambda>0$ such that
\begin{equation}    \label{in:noptcond1}
  \nullv\in\partial\varphi(\bar\xi,\bar x)+
  \lambda\left[\left(\Ncone{\bar\xi}{\Omega}\cap\Uball\right)
  \times\{\nullv\}\right]+{\lambda\over\gamma}\left[\partial\nu(\bar\xi,\bar x)+
  (\Coder{\Kmap}{\bar\xi}{\bar x}(\Uball)\times\Uball)\right].
\end{equation}
\end{theorem}

\begin{proof}
Upon hypotheses (iv), (vi) and (vii), on account of Lemma \ref{lem:erbo1}
it is possible to claim that an uniform error bound around $(\bar\xi,\bar x)$
holds true. This fact, along with hypotheses (i), (ii) and (iii), enables one
apply Proposition \ref{pro:constred}. Then, according to the necessary optimality
condition expressed by $(\ref{in:buncoptcond})$, as a local unconstrained minimizer
of $\varphi+\lambda\dist{\cdot}{\Omega}+\lambda\gamma^{-1}\mf$, $(\bar\xi,\bar x)$
must satisfy the condition
\begin{equation}    \label{in:0inbsudcond}
  \nullv\in\partial\left(\varphi+\lambda\dist{\cdot}{\Omega}
  +{\lambda\over\gamma}\mf\right)(\bar\xi,\bar x).
\end{equation}
Notice that, in the current setting, functions $\varphi+\lambda
\dist{\cdot}{\Omega}$ and ${\lambda\over\gamma}\mf$ are a semi-Lipschitzian
pairs around $(\bar\xi,\bar x)$, while $\varphi$ and $\lambda\dist{\cdot}{\Omega}$
as locally Lipschitz functions clearly fulfil the qualification condition
$(\ref{eq:singsubdqc})$. This allows one
to apply the sume rule for basic subdifferential. Thus from inclusion
$(\ref{in:0inbsudcond})$ and the first relation in $(\ref{eq:bsubddist})$, one obtains
$$
  \nullv\in\partial\varphi(\bar\xi,\bar x)+\lambda\left[\left(\Ncone{\bar\xi}{\Omega}
  \cap\Uball\right)\times\{\nullv\}\right]+
  {\lambda\over\gamma}\partial\mf(\bar\xi,\bar x).
$$
Now, by virtue of Lemma \ref{lem:bsubdmi}, as in this case it is
$\Proj{\bar x}{\Kmap(\bar\xi)}=\{\bar x\}$, from $(\ref{in:bsubdifmi})$
one gets the simpler outer estimate
$$
  \partial\mu(\bar\xi,\bar x)\subseteq\Coder{\Kmap}{\bar\xi}{\bar x}
  (\Uball)\times\Uball.
$$
As a consequence, the qualification condition in hypothesis (v)
implies
$$
  \partial^\infty\nu(\bar\xi,\bar x)\cap\left(
  -\partial\mu(\bar\xi,\bar x)\right)=\{\nullv\},
$$
which allows one to write
$$
  \partial\mf(\bar\xi,\bar x)\subseteq\partial\nu(\bar\xi,\bar x)
  +\partial\mu(\bar\xi,\bar x)\subseteq\partial\nu(\bar\xi,\bar x)
  +(\Coder{\Kmap}{\bar\xi}{\bar x}(\Uball)\times\Uball).
$$
The last inclusions lead obviously to the condition in
the thesis, thereby completing the proof.
\end{proof}

As a comment to the optimality condition emerging from Theorem \ref{thm:nocond1},
one may say that it takes the typical form of a stationarity condition
involving subgradients of the objective functions and (through
normal cones and coderivatives) of data defining the inner equilibrium problem
(compare e.g. with \cite[Theorem 5.49]{Mord06b}, where nonetheless
equilibrium constraints are formalized by parameterized generalized equations
different from $(\VEP(\xi))$).

It is worth noticing that condition $(\ref{in:noptcond1})$ is expressed
in terms of initial problem data, except for the appearance of
function $\nu$, which can be calculated using problem data.

\begin{remark}   \label{rem:transvsufc}
At the price of a minor generality, the geometric assumption (iii) can be
replaced with the following qualification condition, which is expressed
in terms of nonsmooth analysis construction
$$
  \Ncone{\bar\xi}{\Omega}\cap[-\Coder{\Equi}{\bar\xi}{\bar x}(\nullv)]
  =\{\nullv\}.   \leqno {\rm (iii')}
$$
Indeed, as it has been remarked in Section \ref{Sect:2}, according
to $(\ref{eq:subtransvnconecond})$ the subtransversality of $\widetilde{\Omega}$
and $\graph\Equi$ at $(\bar\xi,\bar x)$ is ensured by the condition
$$
  \Ncone{(\bar\xi,\bar x)}{\widetilde{\Omega}}\cap
  [-\Ncone{(\bar\xi,\bar x)}{\graph\Equi}]=\{\nullv\}.
$$
Since it is $\Ncone{(\bar\xi,\bar x)}{\widetilde{\Omega}}=\Ncone{\bar\xi}{\Omega}
\times\Ncone{\bar x}{\R^n}=\Ncone{\bar\xi}{\Omega}\times\{\nullv\}$, and by
definition of coderivative, if $(u,v)\in -\Ncone{(\bar\xi,\bar x)}{\graph\Equi}$
then $-u\in\Coder{\Equi}{\bar\xi}{\bar x}(v)$, the above condition turns
out to be valid provided that, whenever $(u,v)\in\Ncone{\bar\xi}{\Omega}\times\{\nullv\}$
and $-u\in\Coder{\Equi}{\bar\xi}{\bar x}(v)$, one has $(u,v)=\nullv$, that
is $v=\nullv$ and $\Ncone{\bar\xi}{\Omega}\cap[-\Coder{\Equi}{\bar\xi}{\bar x}(\nullv)]
=\{\nullv\}$.
Being not formulated directly on the problem data is a drawback of both (iii)
and (iii'). So it is helpful to note that (iii) is automatically satisfied
if $\bar\xi\inte\Omega$ or in the case $\Omega$ and $\graph\Equi$ are polyhedral.
Besides, the Aubin property of $\Equi$ around $(\bar\xi,\bar x)$ can serve
as a further condition ensuring (iii'), inasmuch is equivalent to
$\Coder{\Equi}{\bar\xi}{\bar x}(\nullv)=\{\nullv\}$
(see \cite[Theorem 3.3(iii)]{Mord18}).
\end{remark}

Below an example illustrates a problem case for which Theorem \ref{thm:nocond1}
can be applied.

\begin{example}
Letting $p=n=1$ and $m=2$, consider a family of vector equilibrium
problems $(\VEP(\xi))$ defined by $f:\R\times\R\times\R\longrightarrow\R^2$,
where
$$
  f(\xi,x,z)=\left(\begin{array}{c}
                     x-z  \\
                     |\xi|
                   \end{array}\right),
$$
$C=\R^2_+=[0,+\infty)\times[0,+\infty)$, and by $\Kmap:\R\rightrightarrows\R$,
where
\begin{equation}    \label{eq:defconstvepmap}
   \Kmap(\xi)=\{x\in\R:\ |x|\le|\xi|+1\}=[-|\xi|-1,|\xi|+1].
\end{equation}
As one readily checks, the associated solution mapping $\Equi:\R\rightrightarrows\R$
is given by
$$
  \Equi(\xi)=\{|\xi|+1\},\quad\forall \xi\in\R.
$$
Now consider the $(\MPEC)$ defined by the above equilibrium constraint,
along with an objective function $\varphi:\R\times\R\longrightarrow\R$
$$
\varphi(\xi,x)=\xi^2+x^2,
$$
and $\Omega=[0,+\infty)$. By direct inspection of the level sets of $\varphi$,
it is plain to see that this $(\MPEC)$ admits as a unique (global)
solution the pair $(\bar\xi,\bar x)=(0,1)$.
Let us show that such a $(\MPEC)$ is suitable for applying Theorem
\ref{thm:nocond1}.

As a smooth function, $\varphi$ is locally Lipschitz around $(0,1)$.
Since $\xi\mapsto |\xi|+1$ and $\xi\mapsto -|\xi|-1$ are functions
continuous on $\R$, $\Kmap$ turns out to be continuous on $\R$
as a set-valued mapping.
The pair of sets $\widetilde{\Omega}=[0,+\infty)\times\R$ and
$\graph\Equi$ is subtransversal at $(0,1)$, because
$$
  \Ncone{(0,1)}{\widetilde{\Omega}}=\{(u,0)\in\R^2:\ u\le 0\},
$$
$$
  \Ncone{(0,1)}{\graph\Equi}=\{(u,v)\in\R^2:\ v\le -|u|\}
  \cup\{(u,v)\in\R^2:\ v\le |u|\}
$$
and therefore it holds
$$
  \Ncone{(0,1)}{\widetilde{\Omega}}\cap[-\Ncone{(0,1)}{\graph\Equi}]
  =\{\nullv\}.
$$
Thus, condition $(\ref{eq:subtransvnconecond})$ is fulfilled.
As a continuous function, $f$ is in particular $\R^2_+$-u.s.c.
and, because $\Kmap$ takes compact values, any set $f(\xi,x,\Kmap(\xi))$
is a compact subset of $\R^2$ for every $(\xi,x)\in\R\times\R$, and hence
it is $\R^2_+$-bounded. It remains to check the validity of hypotheses
(v) and (vii), which require to calculate $\nu$. According to the
definition of $f$, one finds
\begin{eqnarray}    \label{eq:calnuex}
   \nu(\xi,x) &=& \sup_{z\in [-|\xi|-1,|\xi|+1]}\dist{\left(\begin{array}{c}
                     x-z  \\
                     |\xi|
                   \end{array}\right)}{\R^2_+}=\left\{
                   \begin{array}{ll}
                    0, & \hbox{ if } x\ge |\xi|+1,  \\
                       \\
                    |\xi|+1-x, & \hbox{ if } x<|\xi|+1
                   \end{array}
                   \right.    \nonumber   \\
              &=& \max\{|\xi|+1-x,\, 0\}.
\end{eqnarray}
Observe that $\nu$ is convex and Lipschitz continuous on $\R\times\R$,
so one has $\partial^\infty\nu(0,1)=\{\nullv\}$, what makes satisfied
the qualification condition in hypothesis (v).
Take an arbitrary pair $(\xi,x)\in\R\times\R$, such that $x\in\R
\backslash\{|\xi|+1\}$, while $\rho$ can be any positive value.
From $(\ref{eq:calnuex})$, one obtains
$$
  \partial_x\nu(\xi,x)=\left\{{\partial\nu\over\partial x}(\xi,x)\right\}=
  \left\{\begin{array}{ll}
                     \{0\}, & \hbox{ if } x>|\xi|+1,  \\
                     \\
                     \{-1\}, & \hbox{ if } x<|\xi|+1.
                   \end{array}  \right.
$$
On the other hand, it holds
$$
   \Ucone{x}{\Kmap(\xi)}=\left\{\begin{array}{ll}
                     \{1\}, & \hbox{ if }\ x>|\xi|+1,  \\
                     \\
                     \{0\}, & \hbox{ if }\ -|\xi|-1<x<|\xi|+1, \\
                     \\
                     {[-1,0]}, & \hbox{ if }\ x=-|\xi|-1 \\
                     \\
                     \{-1\}, & \hbox{ if }\ x<-|\xi|-1,
                     \end{array}  \right.
$$
and consequently
$$
  \partial_x\nu(\xi,x)+\Ucone{x}{\Kmap(\xi)}=\left\{\begin{array}{ll}
                     \{1\}, & \hbox{ if }\ x>|\xi|+1,  \\
                     \\
                     \{-1\}, & \hbox{ if }\ -|\xi|-1<x<|\xi|+1, \\
                     \\
                     {[-2,-1]}, & \hbox{ if }\ x=-|\xi|-1 \\
                     \\
                     \{-2\}, & \hbox{ if }\ x<-|\xi|-1.
                     \end{array}  \right.
$$
This shows that condition $(\ref{eq:erbosubdgammacond})$ in hypothesis
(vii) is satisfied with any $\gamma\in (0,1)$. Thus Theorem \ref{thm:nocond1}
can be actually applied to the problem under examination.

In order to check the validity of inclusion in $(\ref{in:noptcond1})$,
one needs to calculate $\partial\nu(0,1)$ and $\Coder{\Kmap}{0}{1}([-1,1])$.
By setting
$$
  \nu_1(\xi,x)=|\xi|+1-x, \qquad \nu_2(\xi,x)\equiv 0,
$$
and applying the well-known convex subdifferential calculus rules
for the supremum (maximum) of convex functions, as it is
$$
  \nu_1(0,1)=\nu_2(0,1),
$$
one finds
\begin{eqnarray*}
   \partial\nu(0,1) &=& \clco\left[\partial\nu_1(0,1)\cup\partial\nu_2(0,1)\right]=
   \clco\left[([-1,1]\times\{-1\})\cup\{\nullv\}\right] \\
   &=& \clco\{(-1,-1),\, (-1,1),\, (0,0)\}.
\end{eqnarray*}
According to $(\ref{eq:defconstvepmap})$, it is
$$
  \Ncone{(0,1)}{\graph\Kmap}=\{(u,v)\in\R^2:\ v=|u|\},
$$
so one has
$$
  \Coder{\Kmap}{0}{1}(v)=\left\{\begin{array}{ll}
                                 \varnothing, & \hbox{ if } v>0, \\
                                 \{-v,\ v\},   & \hbox{ if } v\le 0.
                               \end{array}  \right.
$$
It follows
$$
  \Coder{\Kmap}{0}{1}([-1,1])\times [-1,1]=[-1,1]\times[-1,1].
$$
Thus, since it is
$$
  \partial\varphi(0,1)=\{\Gder{\varphi}{0,1}\}=\{(0,2)\}
$$
and
$$
  \Ncone{0}{[0,+\infty)}=(-\infty,0],
$$
then by taking $\lambda=\gamma=1/2$ one finds
\begin{eqnarray*}
  (0,0)&=&(0,2)+(0,0)+(0,-1)+(0,-1) \\
  &\in& \partial\varphi(0,1)+{1\over 2}\left[\Ncone{0}{[0,+\infty)}\times\{0\}\right]+
  \partial\nu(0,1)+\Coder{\Kmap}{0}{1}([-1,1])\times [-1,1].
\end{eqnarray*}
It is worth noticing that in the case of the problem under examination
the stationarity condition $(\ref{in:noptcond1})$ excludes (failing
to be satisfied)  any element in the feasible region, which is different
from the solution. To see this, take an arbitrary $(\xi_0,x_0)\in\graph\Equi$,
with $\xi_0>0$. It is clear that $\Ncone{\xi_0}{[0,+\infty)}=\{0\}$.
Since in a proper neighbourhood of $(\xi_0,x_0)$ it holds
$$
  \nu_1(\xi,x)=\xi+1-x,
$$
it results in
$$
   \partial\nu(\xi_0,x_0)=\clco\{(1,-1),\, (0,0)\}.
$$
Since it is
$$
  \Ncone{(\xi_0,x_0)}{\graph\Kmap}=\{(u,v)\in\R^2:\ 
  v=-u,\ u\le 0\},
$$
one finds
$$
  \Coder{\Kmap}{\xi_0}{x_0}=\left\{\begin{array}{ll}
                                 \varnothing, & \hbox{ if } v<0, \\
                                 \{v\},   & \hbox{ if } v\ge 0.
                               \end{array}  \right.
$$
Therefore, it follows
$$
  \Coder{\Kmap}{\xi_0}{x_0}([-1,1])\times [-1,1]=[0,1]\times [-1,1],
$$
whence one obtains
\begin{eqnarray*}    
   \partial\nu(\xi_0,x_0)&+&\Coder{\Kmap}{\xi_0}{x_0}([-1,1])\times [-1,1] \\ 
   &=&\clco\{(0,1),\, (1,1),\, (2,0),\, (2,-2),\, (1,-2),\, (0,-1)\}.
\end{eqnarray*}
As it is
$$
  \partial\varphi(\xi_0,x_0)=\{\Gder{\varphi}{\xi_0,x_0}\}=\{(2\xi_0,2\xi_0+2)\},
$$
with $\xi_0>0$, whereas
$$
 \partial\nu(\xi_0,x_0)+\Coder{\Kmap}{\xi_0}{x_0}([-1,1])\times [-1,1]
\subseteq [0,+\infty)\times\R,
$$
the inclusion
$$
  (0,0)\in \{(2\xi_0,2\xi_0+2)\}+\lambda\{(0,0)\}+{\lambda\over\gamma}
  \biggl[\partial\nu(\xi_0,x_0)+\Coder{\Kmap}{\xi_0}{x_0}([-1,1])\times [-1,1]\biggl]
$$
can not be true, any which way the values of $\lambda,\, \gamma\in(0,+\infty)$
is chosen.
\end{example}

The optimality condition in Theorem \ref{thm:nocond1} requires
no smoothness and no convexity assumptions. In the last part
of the paper an attempt to improve the computational impact of condition
$(\ref{in:noptcond1})$ is made, by imposing specific smoothness and convexity
assumptions on the problem data.
The main gain is the possibility to estimate $\partial\nu(\bar\xi,\bar x)$
in terms of problem data.

In formulating the next lemma, given $\epsilon>0$ it is convenient to
set
$$
  \overline{\ball{\Kmap(\bar\xi)}{\epsilon}}=\{\hat{z}\in\ball{\Kmap(\bar\xi)}{\epsilon}:
  \ \dist{f(\bar\xi,\bar x,\hat{z})}{C}=\sup_{z\in\ball{\Kmap(\bar\xi)}{\epsilon}}
  \dist{f(\bar\xi,\bar x,z)}{C}\}.
$$
In other words, $\overline{\ball{\Kmap(\bar\xi)}{\epsilon}}$ collects
all the farthest points of the set $f(\bar\xi,\bar x,\ball{\Kmap(\xi)}{\epsilon})$
from $C$. If $\Kmap(\bar\xi)$ is bounded (and hence compact by $(\mathcal{A})$)
and function $z\mapsto f(\bar\xi,\bar x,z)$ is continuous, then set
$\overline{\ball{\Kmap(\bar\xi)}{\epsilon}}$ is clearly nonempty.

\begin{lemma}  \label{lem:subdnuapproxest}
With reference to a family of problem $(\VEP(\xi))$, let $(\bar\xi,
\bar x)\in\graph\Equi$ and $\eta>0$. Suppose that:

\begin{itemize}

\item[(i)] $\Kmap:\R^p\rightrightarrows\R^n$ is concave;

\item[(ii)] $\Kmap$ is u.s.c. and $\Kmap(\bar\xi)$ is bounded;

\item[(iii)] $f:\R^p\times\R^n\times\R^n\longrightarrow\R^m$
is $C$-concave;

\item[(iv)] $f\in C^1(O)$, where $O$ is an open set such that
$\ball{\bar\xi}{\eta}\times\ball{\Kmap(\bar\xi)}{\eta}\times\ball{\Kmap(\bar\xi)}{\eta}
\subseteq O$, and $\Gder{f(\cdot,\cdot,z)}{\bar\xi,\bar x}$ is onto;

\item[(v)] each function $z\mapsto f(\xi,x,z)$ is open with a uniform
linear rate $\alpha>0$ on the set $\ball{\Kmap(\bar\xi)}{\eta}$,
for every $(\xi,x)\in\ball{\bar\xi}{\eta}\times\ball{\Kmap(\bar\xi)}{\eta}$;

\item[(vi)] each function $(\xi,x)\mapsto f(\xi,x,z)$ is Lipschitz continuous
on $\ball{\bar\xi}{\eta}\times\ball{\Kmap(\bar\xi)}{\eta}$  with uniform constant
$\ell_f<\alpha$, for every $z\in\ball{\Kmap(\bar\xi)}{\eta}$.

\end{itemize}
Then, it holds
\begin{equation}   \label{in:subdnuestclcoderf}
   \partial\nu(\bar\xi,\bar x)\subseteq\bigcap_{\epsilon\in (0,\eta/2)}
   \clco\left[\bigcup_{z\in\overline{\ball{\Kmap(\bar\xi)}{\epsilon}}}
   \Gder{f(\cdot,\cdot,z)}{\bar\xi,\bar x}^*(\dcone{C}\cap\Uball)
   +\ell_f\Uball\right].
\end{equation}
\end{lemma}

\begin{proof}
Observe first that, as a composition of two Lipschitz continuous
functions, each function $(\xi,x)\mapsto\dist{f(\xi,x,z)}{C}$ is Lipschitz continuous
on $\ball{\bar\xi}{\eta}\times\ball{\Kmap(\bar\xi)}{\eta}$
with uniform constant $\ell_f$, for every $z\in\ball{\Kmap(\bar\xi)}{\eta}$.
This fact implies, in particular,
\begin{eqnarray}    \label{in:useLipcontf}
   \dist{f(\xi,x,z)}{C} &\ge& \dist{f(\bar\xi,\bar x,z)}{C}-
   \ell_f\|(\xi,x)-(\bar\xi,\bar x)\|,   \\  \nonumber
   & & \forall (\xi,x)\in
   \ball{\bar\xi}{\eta}\times\ball{\Kmap(\bar\xi)}{\eta},\
   \forall z\in\ball{\Kmap(\bar\xi)}{\eta}.
\end{eqnarray}
Fix an arbitrary $\epsilon\in (0,\eta/2)$. As a consequence of hypothesis (v),
one has that
$$
  f(\xi,x,z+\epsilon\Uball)\supseteq f(\xi,x,z)+\alpha\epsilon\Uball,
  \quad\forall z\in\ball{\Kmap(\bar\xi)}{\eta},
$$
for every $(\xi,x)\in\ball{\bar\xi}{\eta}\times\ball{\Kmap(\bar\xi)}{\eta}$,
whence
\begin{equation}     \label{in:useunicovpz}
  f(\xi,x,S+\epsilon\Uball)\supseteq f(\xi,x,S)+\alpha\epsilon\Uball,
  \quad\forall (\xi,x)\in\ball{\bar\xi}{\eta}\times\ball{\Kmap(\bar\xi)}{\eta},
\end{equation}
provided that $S\subseteq\ball{\Kmap(\bar\xi)}{\eta}$. Notice that if it is
$$
  f(\xi,x,z)\in C,\quad\forall z\in\Kmap(\xi),
  \quad\forall (\xi,x)\in\ball{\bar\xi}{\eta}\times
  \ball{\Kmap(\bar\xi)}{\eta},
$$
then one has
$$
  \nu(\xi,x)=0,\quad\forall (\xi,x)\in\ball{\bar\xi}{\eta}\times
  \ball{\bar x}{\eta}.
$$
This implies $\partial\nu(\bar\xi,\bar x)=\{\nullv\}$, so formula
$(\ref{in:subdnuestclcoderf})$ comes true in the case all
$(\xi,x)\in\ball{\bar\xi}{\eta}\times\ball{\Kmap(\bar\xi)}{\eta}$
are such that $f(\xi,x,\Kmap(\xi))\subseteq C$, because $\nullv
\in\dcone{C}\cap\Uball$ and hence, for any $\epsilon>0$, it is
$$
  \nullv\in\bigcup_{z\in\overline{\ball{\Kmap(\bar\xi)}{\epsilon}}}
   \Gder{f(\cdot,\cdot,z)}{\bar\xi,\bar x}^*(\dcone{C}\cap\Uball)
   +\ell_f\Uball.
$$
Otherwise, if for some $(\xi,x)$ it is $f(\xi,x,\Kmap(\xi))\not\subseteq C$, recall that,
given a set $S\not\subseteq C$ and $t>0$,
then it is $\exc{S+t\Uball}{C}=\exc{S}{C}+t$ (see \cite[Lemma 2.2]{Uder19}).
Therefore, taking
\begin{equation}   \label{in:deftepsilon}
 \tepsilon\in\left(\epsilon,\min\left\{{\alpha\over\ell_f},2\right\}\epsilon\right),
\end{equation}
as $\ball{\Kmap(\bar\xi)}{\tepsilon-\epsilon}\subseteq\ball{\Kmap(\bar\xi)}{\eta}$,
from the inclusion $(\ref{in:useunicovpz})$ one obtains
\begin{eqnarray}    \label{in:useopencest}
     \sup_{z\in\ball{\Kmap(\bar\xi)}{\tepsilon}}\dist{f(\xi,x,z)}{C}  \nonumber
     &=& \exc{f(\xi,x,\ball{\Kmap(\bar\xi)}{\tepsilon-\epsilon}+\epsilon\Uball)}{C} \nonumber \\
     &\ge & \exc{f(\xi,x,\ball{\Kmap(\bar\xi)}{\tepsilon-\epsilon})+\alpha\epsilon\Uball}{C}.
\end{eqnarray}
Since $\Kmap$ is u.s.c. at $\bar\xi$, there exists $\delta_\epsilon\in
(0,\epsilon)$ such that
$$
  \Kmap(\xi)\subseteq\Kmap(\bar\xi)+(\tepsilon-\epsilon)\Uball=
  \ball{\Kmap(\bar\xi)}{\tepsilon-\epsilon},
  \quad\forall \xi\in\ball{\bar\xi}{\delta_\epsilon}.
$$
Thus, if $f(\xi,x,\Kmap(\xi))\not\subseteq C$, then a fortiori
it holds
$$
  f(\xi,x,\ball{\Kmap(\bar\xi)}{\tepsilon-\epsilon})\not\subseteq C.
$$
By virtue of this inclusion, from inequality $(\ref{in:useopencest})$ one gets
\begin{eqnarray}    \label{in:nutepsnuest}
  \sup_{z\in\ball{\Kmap(\bar\xi)}{\tepsilon}}\dist{f(\xi,x,z)}{C}
  &\ge & \exc{f(\xi,x,\ball{\Kmap(\bar\xi)}{\tepsilon-\epsilon}}{C}+\alpha\epsilon \nonumber  \\
   &\ge & \sup_{z\in\Kmap(\xi)}\dist{f(\xi,x,z)}{C}+\alpha\epsilon,   \\
  & & \forall (\xi,x)\in\ball{\bar\xi}{\delta_\epsilon}\times\ball{\Kmap(\bar\xi)}{\eta}:
  \ f(\xi,x,\Kmap(\xi))\not\subseteq C.   \nonumber
\end{eqnarray}
Now, let us define the function $\nu_{\tepsilon}:\R^p\times\R^n\times\R^n
\longrightarrow\R\cup\{\pm\infty\}$ by setting
\begin{eqnarray}    \label{eq:defnutepsilon}
  \nu_{\tepsilon}(\xi,x) &=& \sup_{z\in\ball{\Kmap(\bar\xi)}{\tepsilon}}\dist{f(\xi,x,z)}{C}
  -\sup_{z\in\ball{\Kmap(\bar\xi)}{\tepsilon}}\dist{f(\bar\xi,\bar x,z)}{C}  \\
  &+& \ell_f\|(\xi,x)-(\bar\xi,\bar x)\|.   \nonumber
\end{eqnarray}
Concerning $\nu_{\tepsilon}$, the following claims can be made:
$$
  \nu_{\tepsilon}(\bar\xi,\bar x)=0; \leqno(c_1)
$$
$$
  \nu_{\tepsilon}(\xi,x)\ge 0,\quad\forall (\xi,x)
  \in\ball{\bar\xi}{\delta_\epsilon}\times\ball{\bar x}{\delta_\epsilon};
  \leqno(c_2)
$$
$$
  \nu_{\tepsilon}(\xi,x)\ge\nu(\xi,x),\quad\forall (\xi,x)
  \in\ball{\bar\xi}{\delta_\epsilon}\times\ball{\bar x}{\delta_\epsilon}.
  \leqno(c_3)
$$
The validity of claim $(c_1)$ follows at once from $(\ref{eq:defnutepsilon})$.

To show that $(c_2)$ holds true, it suffices to observe that, from
inequality $(\ref{in:useLipcontf})$, as it is $\tepsilon<2\epsilon<\eta$
and $\delta_\epsilon<\epsilon<\eta$, it follows
$$
  \sup_{z\in\ball{\Kmap(\bar\xi)}{\tepsilon}}\dist{f(\xi,x,z)}{C}\ge
  \sup_{z\in\ball{\Kmap(\bar\xi)}{\tepsilon}}\dist{f(\bar\xi,\bar x,z)}{C}
  -\ell_f\|(\xi,x)-(\bar\xi,\bar x)\|
$$
wherefrom, according with the definition in $(\ref{eq:defnutepsilon})$,
one obtains $(c_2)$ immediately.

As for claim $(c_3)$, if $(\xi,x)\in\ball{\bar\xi}{\delta_\epsilon}\times
\ball{\bar x}{\delta_\epsilon}$ are such that $f(\xi,x,\Kmap(\xi))\subseteq C$,
then it is $\nu(\xi,x)=0$, so $(c_2)$ implies $(c_3)$.
For all those $(\xi,x)\in\ball{\bar\xi}{\delta_\epsilon}\times
\ball{\bar x}{\delta_\epsilon}$ such that $f(\xi,x,\Kmap(\xi))\not\subseteq C$,
let us notice that for any $z_1\in\ball{\Kmap(\bar\xi)}{\tepsilon}$
and $z_2\in\Kmap(\bar\xi)$, by virtue of hypothesis
(vi) one has
\begin{eqnarray*}
  \dist{f(\bar\xi,\bar x,z_1)}{C} &\le& \|f(\bar\xi,\bar x,z_1)-
  f(\bar\xi,\bar x,z_2)\|+\dist{f(\bar\xi,\bar x,z_2)}{C}  \\
  &\le& \ell_f\|z_1-z_2\|,
\end{eqnarray*}
whence
$$
  \dist{f(\bar\xi,\bar x,z_1)}{C}\le\inf_{z_2\in\Kmap(\bar\xi)}
  \ell_f\|z_1-z_2\|=\ell_f\dist{z_1}{\Kmap(\bar\xi)}\le\ell_f\tepsilon.
$$
Consequently, it holds
\begin{equation}  \label{in:Kmaptepsilonest}
  \sup_{z_1\in\ball{\Kmap(\bar\xi)}{\tepsilon}}\dist{f(\bar\xi,\bar x,z_1)}{C}
  \le\ell_f\tepsilon.
\end{equation}
Thus, by combining inequalities $(\ref{in:nutepsnuest})$ and
$(\ref{in:Kmaptepsilonest})$, one obtains
$$
  \nu_{\tepsilon}(\xi,x)\ge\nu(\xi,x)+\alpha\epsilon-\ell_f\tepsilon+
  \ell_f\|(\xi,x)-(\bar\xi,\bar x)\|,\quad\forall (\xi,x)\in
  \ball{\bar\xi}{\delta_\epsilon}\times\ball{\Kmap(\bar\xi)}{\eta}.
$$
On account of $(\ref{in:deftepsilon})$, it is $\alpha\epsilon
-\ell_f\tepsilon>0$, so the last inequality proves the validity
of $(c_3)$.

By remembering Lemma \ref{lem:convnu}, one can remark that,
under hypotheses (i) and (iii), both $\nu$ and $\nu_{\epsilon}$
are convex functions. Since it is
$$
  \nu_{\tepsilon}(\bar\xi,\bar x)=\nu(\bar\xi,\bar x)=0
$$
and $(c_3)$ holds true, by a well-known property of the Fenchel
subdifferential one has
\begin{equation}    \label{in:subdnuapprox}
   \partial\nu(\bar\xi,\bar x)\subseteq\partial\nu_{\tepsilon}
   (\bar\xi,\bar x).
\end{equation}
Owing to the structure of $\nu_{\tepsilon}$, the latter
subdifferential can be exactly calculated by means of the known rule
for subdifferential of supremum of convex functions
(see, for instance, \cite[Theorem 2.4.18]{Zali02}). Such a rule
can be applied because by hypothesis $\Kmap(\bar\xi)$ is compact
and hence so is $\ball{\Kmap(\bar\xi)}{\tepsilon}$.
By consequence, as the function $\dist{\cdot}{C}\circ f$
is continuous on $\ball{\bar\xi}{\eta}
\times\ball{\Kmap(\bar\xi)}{\eta}\times\ball{\Kmap(\bar\xi)}{\eta}$, the set
$\overline{\ball{\Kmap(\bar\xi)}{\tepsilon}}$ is nonempty.
The continuity of $\dist{\cdot}{C}\circ f$ makes it possible
to apply also the sum rule for subdifferentials. All of this
results in
\begin{eqnarray}   \label{eq:subdDubMilformula}
   \partial\nu_{\tepsilon}(\bar\xi,\bar x) &=& \partial
   \left(\sup_{z\in\ball{\Kmap(\bar\xi)}{\tepsilon}}\dist{f(\cdot,\cdot,z)}{C}
  -\sup_{z\in\ball{\Kmap(\bar\xi)}{\tepsilon}}\dist{f(\bar\xi,\bar x,z)}{C}
  \right)(\bar\xi,\bar x)  \nonumber \\
  &+& \ell_f\partial\|\cdot-(\bar\xi,\bar x)\|(\bar\xi,\bar x) \nonumber \\
  &=& \clco\bigcup_{z\in\overline{\ball{\Kmap(\bar\xi)}{\tepsilon}}}
  \partial\dist{f(\cdot,\cdot,z)}{C}(\bar\xi,\bar x)+\ell_f\Uball.
\end{eqnarray}
By hypothesis (iv) each function $(\xi,x)\mapsto f(\xi,x,z)$
is strictly differentiable at $(\bar\xi,\bar x)$, for every $z\in
\ball{\Kmap(\bar\xi)}{\tepsilon}$, with the derivative
$\Gder{f(\cdot,\cdot,z)}{\bar\xi,\bar x}$ being onto.
Then it is possible to apply the formula in \cite[Proposition 1.112(i)]{Mord06a},
according to which
\begin{equation}     \label{eq:subdcompfunct}
  \partial\dist{f(\cdot,\cdot,z)}{C}(\bar\xi,\bar x)=
  \Gder{f(\cdot,\cdot,z)}{\bar\xi,\bar x}^*
  \left(\partial\dist{\cdot}{C}(f(\bar\xi,\bar x,z)\right).
\end{equation}
As $C$ is a convex cone, function $\dist{\cdot}{C}$ is p.h. and convex,
i.e. sublinear. From this fact, along with Lipschitz continuity with constant
$1$, one readily sees that
$$
  \partial\dist{\cdot}{C}(y)\subseteq\dcone{C}\cap\Uball,\quad
  \forall y\in\R^m.
$$
By using the last inclusion in equality $(\ref{eq:subdcompfunct})$,
one finds
$$
 \partial\dist{f(\cdot,\cdot,z)}{C}(\bar\xi,\bar x)\subseteq
 \Gder{f(\cdot,\cdot,z)}{\bar\xi,\bar x}^*
 \left(\dcone{C}\cap\Uball\right).
$$
By taking into account $(\ref{eq:subdDubMilformula})$, the last
estimate gives
$$
   \partial\nu_{\tepsilon}(\bar\xi,\bar x)\subseteq \clco
   \bigcup_{z\in\overline{\ball{\Kmap(\bar\xi)}{\tepsilon}}}
   \Gder{f(\cdot,\cdot,z)}{\bar\xi,\bar x}^*\left(\dcone{C}\cap\Uball\right)
   +\ell_f\Uball,
$$
and hence, in the light of inclusion $(\ref{in:subdnuapprox})$,
$$
  \partial\nu(\bar\xi,\bar x)\subseteq \clco
   \bigcup_{z\in\overline{\ball{\Kmap(\bar\xi)}{\tepsilon}}}
   \Gder{f(\cdot,\cdot,z)}{\bar\xi,\bar x}^*\left(\dcone{C}\cap\Uball\right)
   +\ell_f\Uball.
$$
Since if taking arbitrarily $\epsilon\in (0,\eta/2)$, then $\tepsilon$ chosen
as in $(\ref{in:deftepsilon})$ can cover $(0,\eta/2)$, the last inclusion
leads to the formula in the thesis.
\end{proof}

The next lemma provides a parametric error bound for the solutions to
problem $(\VEP(\xi))$ in the case the problem data satisfy special assumptions
about smoothness and concavity.
Consistently with the notation previously introduced, let us set
$$
  \overline{\Kmap(\xi)}=\{\hat z\in\Kmap(\xi):\
  \dist{f(\xi,x,\hat z)}{C}=\sup_{z\in\Kmap(\xi)}
  \dist{f(\xi,x,z)}{C} \}.
$$

\begin{lemma}[Error bound under smoothness and concavity]     \label{lem:perbosmconc}
Let $\rho>0$. With reference to a problem $(\VEP(\xi))$, with
$\xi\in\ball{\bar\xi}{\rho}$, let $(\bar\xi,\bar x)\in\graph\Equi$.
Suppose that:

\begin{itemize}

\item[(i)] $\Kmap$ takes convex compact values on $\ball{\bar\xi}{\rho}$;

\item[(ii)] $f$ is $C$-concave;

\item[(iii)] the mapping $x\mapsto f(\xi,\cdot,z)\in C^1(\R^n)$, with
$\Gder{f(\xi,\cdot,z)}{x}$ onto, for every $\xi\in\ball{\bar\xi}{\rho}$, $x\in\R^n
\backslash\Equi(\xi)$ and $z\in\Kmap(\xi)$;

\item[(iv)] there exists $\gamma>0$ such that
\begin{eqnarray*}
  \left[\clco\bigcup_{z\in\overline{\Kmap(\xi)}}\Gder{f(\xi,\cdot,z)}{x}^*
  \left({f(\xi,x,z)-\Proj{f(\xi,x,z)}{C}\over\dist{f(\xi,x,z)}{C}}\right)+
  \Ucone{x}{\Kmap(\xi)}\right]\cap
  \gamma\Uball=\varnothing, \\
  \forall x\in\R^n\backslash
  \Equi(\xi),\ \forall\xi\in\ball{\bar\xi}{\rho}.
\end{eqnarray*}
\end{itemize}
Then, $\Equi(\xi)\ne\varnothing$ and it holds
\begin{equation}
   \dist{x}{\Equi(\xi)}\le{\mf(\xi,x)\over \gamma},
   \quad\forall x\in\R^n.
\end{equation}
\end{lemma}

\begin{proof}
Fix $\xi\in\ball{\bar\xi}{\rho}$ and consider the function $x\mapsto\mf(\xi,x)$.
By virtue of hypothesis (iii) each function $x\mapsto \dist{f(\xi,x,z)}{C}$
is continuous for every $z\in\Kmap(\xi)$, so the supremum over $\Kmap(\xi)$
is l.s.c..  As function $x\mapsto\dist{x}{\Kmap(\xi)}$ is (Lipschitz)
continuous, it is clear that $\mf(\xi,\cdot)$ is l.s.c. on $\R^n$.
Observe that, since $\Kmap(\xi)$ is compact by hypothesis (i), one has
$\dom\mf(\xi,\cdot)=\R^n=[\mf(\xi,\cdot)<+\infty]\ne\varnothing$.
Besides, by virtue of hypothesis (iii) and the convexity of $\Kmap(\xi)$,
$\mf(\xi,\cdot)$ turns out to be convex, so actually continuous on $\R^n$.
In such a circumstance, the estimate
\begin{equation}    \label{eq:stslestconvexcase}
  \stsl{\mf(\xi,\cdot)}(x)=\dist{\nullv}{\partialx\mf(\xi,x)}
\end{equation}
is known to hold (see, for instance, \cite[Theorem 5(i)]{FaHeKrOu10}).
Under the aforementioned continuity properties of $\nu(\xi,\cdot)$
and $\mu(\xi,\cdot)$, it is possible to write
$$
  \partialx\mf(\xi,x)=\partialx\nu(\xi,x)+\partialx\mu(\xi,x)
  \subseteq\partialx\nu(\xi,x)+\Ucone{x}{\Kmap(\xi)},
$$
where the unit truncation map defined as in $(\ref{eq:deftruncmap})$
is now constructed by means of
the normal cone in the sense of convex analysis. By proceeding
as in the proof of Lemma \ref{lem:subdnuapproxest} one finds
\begin{eqnarray}   \label{eq:parsubdnuclconvGderf}
    \partialx\nu(\xi,x)&=&\clco\bigcup_{z\in\overline{\Kmap(\xi)}}
    \partial\left(\dist{f(\xi,\cdot,z)}{C}\right)(x) \nonumber  \\
    &=&\clco\bigcup_{z\in\overline{\Kmap(\xi)}}
    \Gder{f(\xi,\cdot,z)}{x}^*
    \left(\partial\dist{\cdot}{C}(f(\xi,x,z))\right).
\end{eqnarray}
Now, if $x\in\R^n\backslash\Equi(\xi)$ and $z\in\overline{\Kmap(\xi)}$,
it means that $f(\xi,x,z)\not\in C$. Since function $y\mapsto\dist{y}{C}$
is strictly differentiable on $\R^m\backslash C$ (thanks to the Euclidean
space structure),  according to $(\ref{eq:bsubdoosdist})$ one finds
$$
  \partial\dist{\cdot}{C}(f(\xi,x,z))=\left\{\Gder{\dist{\cdot}{C}}{f(\xi,x,z)}\right\}=
  \left\{{f(\xi,x,z)-\Proj{f(\xi,x,z)}{C}\over\dist{f(\xi,x,z)}{C}}\right\}.
$$
By employing this subdifferential representation in the second
equality of $(\ref{eq:parsubdnuclconvGderf})$, one obtains
\begin{eqnarray*}
   \partialx\nu(\xi,x)=\clco\bigcup_{z\in\overline{\Kmap(\xi)}}
   \Gder{f(\xi,\cdot,z)}{x}^*\left({f(\xi,x,z)-\Proj{f(\xi,x,z)}{C}
   \over\dist{f(\xi,x,z)}{C}}\right) \\
   \forall x\in\R^n\backslash\Equi(\xi),\
   \forall\xi\in\ball{\bar\xi}{\rho} .
\end{eqnarray*}
In the light of the estimate $(\ref{eq:stslestconvexcase})$, the last
equality, along with hypothesis (iv), implies that
$$
  \stsl{\mf(\xi,\cdot)}(x)\ge\gamma,\quad\forall x\in\R^n
  \backslash\Equi(\xi).
$$
By invoking a well-known condition for the error bound of
convex functions (see, for instance, \cite[Proposition 3.1]{Uder22}),
the last inequality guarantees both the assertions in the
thesis.
\end{proof}

The preceding lemmata single out a setting where the stationary
condition in $(\ref{in:noptcond1})$ can be fully formulated in
terms of initial problem data.

\begin{theorem}[Necessary optimality condition under smoothness
and concavity]
Let $(\bar\xi,\bar x)\in\graph\Equi$, with $\bar\xi\in\Omega$, be a local
solution to $(\MPEC)$ and let $\rho>0$. Suppose that:

\begin{itemize}

\item[(i)] $\varphi$ is locally Lipschitz around $(\bar\xi,\bar x)$;

\item[(ii)] it holds  $\Ncone{\bar\xi}{\Omega}\cap
[-\Coder{\Equi}{\bar\xi}{\bar x}(\nullv)]=\{\nullv\}$;

\item[(iii)] $\Kmap:\R^p\rightrightarrows\R^n$ is concave
and takes convex values on $\ball{\bar\xi}{\rho}$;

\item[(iv)] $\Kmap$ is continuous and $\Kmap(\bar\xi)$ is bounded;

\item[(v)] $f:\R^p\times\R^n\times\R^n\longrightarrow\R^m$
is $C$-concave;

\item[(vi)] $f\in C^1(O)$, where $O$ is an open set such that
$\ball{\bar\xi}{\rho}\times\R^n\times\ball{\Kmap(\bar\xi)}{\rho}
\subseteq O$, with $\Gder{f(\cdot,\cdot,z)}{\bar\xi,\bar x}$
and $\Gder{f(\xi,\cdot,z)}{x}$ onto;

\item[(vii)] each function $z\mapsto f(\xi,x,z)$ is open with a uniform
linear rate $\alpha>0$ on the set $\ball{\Kmap(\bar\xi)}{\rho}$,
for every $(\xi,x)\in\ball{\bar\xi}{\rho}\times\ball{\Kmap(\bar\xi)}{\rho}$;

\item[(viii)] each function $(\xi,x)\mapsto f(\xi,x,z)$ is Lipschitz continuous
on $\ball{\bar\xi}{\rho}\times\ball{\Kmap(\bar\xi)}{\rho}$  with uniform constant
$\ell_f<\alpha$, for every $z\in\ball{\Kmap(\bar\xi)}{\rho}$;

\item[(ix)] there exists $\gamma>0$ such that
\begin{eqnarray*}
  \left[\clco\bigcup_{z\in\overline{\Kmap(\xi)}}\Gder{f(\xi,\cdot,z)}{x}^*
  \left({f(\xi,x,z)-\Proj{f(\xi,x,z)}{C}\over\dist{f(\xi,x,z)}{C}}\right)+
  \Ucone{x}{\Kmap(\xi)}\right]\cap
  \gamma\Uball=\varnothing, \\
  \forall x\in\R^n\backslash
  \Equi(\xi),\ \forall\xi\in\ball{\bar\xi}{\rho}.
\end{eqnarray*}
\end{itemize}
Then, there exists $\lambda>0$ such that
\begin{eqnarray*}
  \nullv\in\partial\varphi(\bar\xi,\bar x) &+&
  \lambda\left[\left(\Ncone{\bar\xi}{\Omega}\cap\Uball\right)
  \times\{\nullv\}\right]  \\ &+& {\lambda\over\gamma} \biggl\{
  \bigcap_{\epsilon\in (0,\rho/2)}
   \clco\biggl[\bigcup_{z\in\overline{\ball{\Kmap(\bar\xi)}{\epsilon}}}
   \Gder{f(\cdot,\cdot,z)}{\bar\xi,\bar x}(\dcone{C}\cap\Uball)
   +\ell_f\Uball\biggl]  \\
   &+&  [\Coder{\Kmap}{\bar\xi}{\bar x}(\Uball)\times\Uball]\biggl\}.
\end{eqnarray*}
\end{theorem}

\begin{proof}
The first part of the proof consists in showing that, under the
current apparatus of hypotheses, it is possible to apply Theorem
\ref{thm:nocond1}. In the second part, the condition in the thesis
is easily derived from $(\ref{in:noptcond1})$, by exploiting the
outer estimate of $\partial\nu(\bar\xi,\bar x)$ provided by Lemma
\ref{lem:subdnuapproxest}.

Observe that by the concavity of $\Kmap$ and the $C$-concavity
of $f$, function $\nu$ is convex, while by the continuity of $\Kmap$
and the smoothness of $f$ function $\nu$ is l.s.c.. By the upper semicontinuity
of $\Kmap$ at $\bar\xi$ there exists $\delta\in (0,\rho)$ such that
$$
  \Kmap(\xi)\subseteq\ball{\Kmap(\bar\xi)}{\rho},\quad\forall
  \xi\in\ball{\bar\xi}{\delta}.
$$
As $\Kmap(\bar\xi)$ is compact, the above inclusion implies that
$\Kmap$ takes compact values around $\bar\xi$.
Consequently, by continuity on $\ball{\Kmap(\bar\xi)}{\rho}$ of
each function $f(\xi,x,\cdot)$, with $(\xi,x)\in\ball{\bar\xi}{\rho}
\times\ball{\Kmap(\bar\xi)}{\rho}$, one has that each set
$f(\xi,x,\Kmap(\xi))$ is compact, for any $(\xi,x)\in\ball{\bar\xi}{\rho}
\times\ball{\bar x}{\rho}$. It follows that $(\bar\xi,\bar x)\in
\inte\dom\nu$. Thus, as a convex function, $\nu$ is locally Lipschitz
around $(\bar\xi,\bar x)$. This fact entails that
$
  \partial^\infty\nu(\bar\xi,\bar x)=\{\nullv\},
$
thereby showing that the qualification condition
$$
  -\partial^\infty\nu(\bar\xi,\bar x)\cap
  \left(\Coder{\Kmap}{\bar\xi}{\bar x}(\Uball)\times
  \Uball\right)=\{\nullv\}
$$
is satisfied as required in Theorem \ref{thm:nocond1}.
As commented in Remark \ref{rem:transvsufc}, hypothesis (ii)
ensures that $\widetilde{\Omega}$ and $\graph\Equi$ are subtransversal
at $(\bar\xi,\bar x)$.
Besides, the existence for every $\xi\in\ball{\bar\xi}{\rho}$ of $x_{0,\xi}
\in\Kmap(\xi)$ such that $f(\xi,x_{0,\xi},\Kmap(\xi))$  is $C$-bounded becomes a
consequence of the compactness of $f(\xi,x,\Kmap(\xi))$, for every
$x\in\ball{\bar x}{\delta}$.
All the other hypotheses of Lemma \ref{lem:perbosmconc} being fulfilled,
the validity of condition $(\ref{eq:erbosubdgammacond})$ is ensured
by virtue of hypothesis (ix).
Thus the local optimality of $(\bar\xi,\bar x)$ leads to the existence
of $\lambda>0$ such that the inclusion in $(\ref{in:noptcond1})$
holds.
It remains to notice that hypotheses (iii)-(viii) make it possible
to apply Lemma \ref{lem:subdnuapproxest} in such a way to
express $\partial\nu(\bar\xi,\bar x)$ as in formula $(\ref{in:subdnuestclcoderf})$.
This leads to the condition in the assertion.
\end{proof}

\vskip1cm


\end{document}